\documentclass[10pt,leqno]{amsart}%
\usepackage{amsmath}
\usepackage{amsfonts}
\usepackage{amssymb}
\usepackage{graphicx}%
\theoremstyle{plain}
\newtheorem{theorem}{Theorem}
\newtheorem{proposition}[theorem]{Proposition}
\newtheorem{corollary}[theorem]{Corollary}
\newtheorem{lemma}[theorem]{Lemma}

\numberwithin{equation}{section}

\theoremstyle{definition}
\newtheorem{definition}[theorem]{Definition}

\newtheorem{remark}[theorem]{Remark}

\newcommand{\define}[1]{{\em #1\/}}

\newcommand{\cA}{{\mathcal A}}

\newcommand{\cC}{{\mathcal C}}

\newcommand{\cE}{{\mathcal E}}
\newcommand{\cF}{{\mathcal F}}

\newcommand{\cM}{{\mathcal M}}

\newcommand{\cP}{{\mathcal P}}

\newcommand{\cU}{{\mathcal U}}

\newcommand{\cW}{{\mathcal W}}

\newcommand{\field}[1]{\mathbb{#1}}
\newcommand{\N}{\field{N}}          		
\newcommand{\R}{\field{R}}          		
\newcommand{\Sph}{\field{S}}        		

\newcommand{\abs}[1]{\lvert #1 \rvert}

\newcommand{\norm}[1]{\lVert #1 \rVert}

\newcommand{\ang}[1]{\langle #1 \rangle}

\newcommand{\loc}{{\rm loc}}

\DeclareMathOperator\diver{div}

\DeclareMathOperator\card{card}

\DeclareMathOperator\supp{supp}

\DeclareMathOperator\vol{Vol}

\DeclareMathOperator\sect{Sect}
\DeclareMathOperator\ric{Ric}
\DeclareMathOperator\riem{Riem}
\DeclareMathOperator\sgn{sgn}
\DeclareMathOperator\osc{osc}
\DeclareMathOperator*{\esssup}{ess\,sup}
\begin{document}
\title[Asymptotic Dirichlet problem]{Asymptotic Dirichlet problem for $\cA$-harmonic and minimal graph 
equations in Cartan-Hadamard manifolds}
\author[Jean-Baptiste Casteras]{Jean-Baptiste Casteras}
\address{Departement de Mathematique
Universite libre de Bruxelles, CP 214, Boulevard du Triomphe, B-1050 Bruxelles, Belgium.}
\email{jeanbaptiste.casteras@gmail.com}
\author[Ilkka Holopainen]{Ilkka Holopainen}
\address{Department of Mathematics and Statistics, P.O. Box 68, 00014 University of
Helsinki, Finland.}
\email{ilkka.holopainen@helsinki.fi}
\author[Jaime B. Ripoll]{Jaime B. Ripoll}
\address{UFRGS, Instituto de Matem\'atica, Av. Bento Goncalves 9500, 91540-000 Porto
Alegre-RS, Brasil.}
\email{jaime.ripoll@ufrgs.br}
\thanks{J.-B.C. supported by the CNPq (Brazil) project 501559/2012-4; I.H. supported by the Academy of 
Finland, project 252293; J.R. supported by the CNPq (Brazil) project 302955/2011-9}

\begin{abstract}
We study the asymptotic Dirichlet problem for $\cA$-harmonic equations and for the minimal graph equation 
on a Cartan-Hadamard manifold $M$ whose sectional curvatures are bounded from below and above by certain 
functions depending on the distance $r=d(\cdot,o)$ to a fixed point $o\in M$. We are, in particular,  
interested in finding optimal (or close to optimal) curvature upper bounds. In the special case of the 
Laplace-Beltrami equation we are able to solve the asymptotic Dirichlet problem in dimensions $n\ge 3$ if 
radial sectional curvatures satisfy
\[
-\frac{\bigl(\log r(x)\bigr)^{2\bar{\varepsilon}}}{r(x)^2} \leq K \leq -\frac
{1+\varepsilon}{r(x)^{2}\log r(x)}
\]
outside a compact set for some $\varepsilon > \bar{\varepsilon}>0$. 
The upper bound is close to optimal since the nonsolvability is known if
$K\geq -1/\bigl(2r(x)^2 \log r(x)\bigr)$. Our results (in the non-rotationally symmetric case) 
improve on 
the previously known case of the quadratically decaying upper bound.
\end{abstract}
\maketitle
\bibliographystyle{acm}

\section{Introduction}\label{sec_intro}
In this paper we are interested in the asymptotic Dirichlet problem for $\cA$-harmonic functions and for 
the minimal graph equation on  a Cartan-Hadamard manifold $M$ of dimension $n\ge 2$. We first recall that 
a Cartan-Hadamard manifold is a simply connected, complete Riemannian manifold having nonpositive 
sectional curvature. It is well-known, since the exponential map $\exp_{o}\colon T_{o}M\to M$ is a 
diffeomorphism for every point $o\in M$, that $M$ is diffeomorphic to $\mathbb{R}^{n}$. One can define an 
asymptotic boundary $\partial_\infty M$ of $M$ as the set of all equivalence classes of unit speed 
geodesic rays on $M$ (see Section~\ref{subsec_hada} for more details). The so-called geometric 
compactification 
$\bar{M}$ of $M$ is then given by $\bar{M}=M\cup \partial_\infty M$ equipped with the 
\emph{cone topology\/}. We also notice that $\bar{M}$ is homeomorphic to a closed Euclidean ball; see 
\cite{EO}. The \emph{asymptotic Dirichlet problem\/} on $M$ for some operator ${\mathcal{Q}}$ is then 
the following: Given a continuous function $f$ on $\partial_\infty M$ does there exist a (unique)
function $u\in C( \bar{M})$ such that ${\mathcal{Q}}[u]=0$ on $M$ and $u\vert\partial_\infty M=f$?
We will consider this problem for two kinds of operators: the minimal graph operator (or the mean 
curvature operator) $\cM$ defined by 
\[
\cM [u]=\diver \dfrac{\nabla u}{\sqrt{1+|\nabla u|^2}},
\]
and the $\cA$-harmonic operator (of type $p$)
\begin{equation}\label{Aoper}
\mathcal{Q}[u]= -\diver \cA_x(\nabla u),
\end{equation}
where $\cA\colon TM\to TM$ is subject to certain conditions; for instance
$$\langle{\mathcal{A}}(V),V\rangle\approx\lvert V \rvert^{p},\quad 1<p<\infty,$$
and
${\mathcal{A}}(\lambda V)=\lambda\lvert\lambda\rvert^{p-2}{\mathcal{A}}(V)$
for all $\lambda\in\mathbb{R}\setminus\{0\}$. The $p$-Laplacian is an example of an $\cA$-harmonic 
operator (see Section~\ref{subsec_A-harm}  for the precise definition). We also note that $u$ satisfies 
the minimal graph equation
\begin{equation}\label{Mequ}
\cM [u]=\diver \dfrac{\nabla u}{\sqrt{1+|\nabla u|^2}}=0
\end{equation} 
if and only if $G:=\left\{  \left(  x,u(x)\right)|x\in\Omega\right\}$  is a minimal 
hypersurface in the product space $M\times\mathbb{R}$.

We will now give a brief overview of the results known for the asymptotic Dirichlet problem on 
Cartan-Hadamard manifolds. The first result for this problem in the case of the usual Laplace-Beltrami 
operator was obtained by Choi. In \cite{choi}, he solved the asymptotic Dirichlet 
problem assuming that the sectional curvatures
satisfy $K\leq-a^{2}<0$ and that $M$ satisfies a ``convex conic neighborhood condition", i.e. given 
$x\in \partial_\infty M$, for any $y\in \partial_\infty M$, $y\neq x$, there exist 
$V_x\subset \bar{M}$, a neighborhood of $x$, and $V_y\subset \bar{M}$, a neighborhood of $y$ 
such that $V_x$ and $V_y$ are disjoint open sets of $\bar{M}$ in terms of the cone topology and 
$V_x \cap M$ is convex with $C^2$ boundary. Anderson \cite{andJDG} proved that the convex conic 
neighborhood condition is satisfied for manifolds of pinched sectional 
curvature $-b^{2}\leq K\leq-a^{2}<0$ and therefore he was able to solve the asymptotic Dirichlet problem 
for the Laplace-Beltrami operator (see also \cite{andschoen} for a different approach). We point out that 
the asymptotic Dirichlet problem was solved independently by Sullivan \cite{sullivan} using probabilistic 
arguments. Ancona in a series of papers \cite{ancannals}, \cite{anchyp}, \cite{ancpot}, and 
\cite{ancrevista}, was able to replace the curvature lower bound by a bounded geometry assumption
that each ball up to a fixed radius is $L$-bi-Lipschitz equivalent to an open set in $\mathbb{R}^{n}$ 
for some fixed $L\geq1$; see \cite{ancannals}. 
To the best of our knowledge, the most general curvature bounds so far under which the asymptotic Dirichlet problem for the Laplacian  was known to be solvable in all dimensions $n\ge 2$ are given in the following theorem by Hsu.
Here and throughout the paper $r(x)$ stands for the distance between $x\in M$ and a fixed point 
$o\in M$.

\begin{theorem}
\cite[Theorems 1.1 and 1.2]{Hs}\label{HsuThm1} Let $M$ be a Cartan-Hadamard manifold.
Suppose that:\\
- there exist a positive constant $a$ and a positive and
non-increasing function $h$ with $\int_{0}^{\infty}t h(t)\,dt<\infty$ such
that
\[
-h\bigl(r(x)\bigr)^{2}e^{2ar(x)}\le\ric_{x}\quad\text{and}\quad
K_x\le-a^{2},
\]
or\\
- there exist positive constants $r_{0},\ \alpha>2,$ and
$\beta<\alpha-2$ such that
\[
-r(x)^{2\beta}\le\ric_{x}\quad\text{and}\quad K_{x}\le-\frac
{\alpha(\alpha-1)}{r(x)^{2}}
\]
for all $x\in M$, with $r(x)\ge r_{0}$.
Then the Dirichlet problem at infinity for the Laplacian is solvable.
\end{theorem}

The asymptotic Dirichlet problem has been studied for more general operators than the Laplacian. The first 
result in this direction has been obtained in \cite{Ho} for the $p$-Laplacian under a pinched negative 
sectional curvature assumption by modifying the direct approach of Anderson and Schoen \cite{andschoen}. 
In \cite{HoVa} Holopainen and V\"{a}h\"{a}kangas have been able to relax the assumption on the curvature 
(see Theorem \ref{HVkor2_RT} for a more precise statement of these curvature assumptions). Of particular 
interest is the case of the minimal graph operator. In \cite{CR}, Collin and Rosenberg were able to 
construct harmonic diffeomorphisms from the complex plane $\mathbb{C}$ onto the hyperbolic plane 
$\mathbb{H}^2$ disproving this way a conjecture of Schoen and Yau \cite{ScYau}. This result has been 
generalized by G{\'a}lvez and Rosenberg \cite{GR} to any Hadamard surface $M$ whose curvature is bounded 
from above by a negative constant. 
A fundamental ingredient in their constructions is to solve the Dirichlet problem on unbounded ideal 
polygons with boundary values $\pm\infty$ on the sides of the ideal polygons. These unexpected results 
have raised interest in (entire) minimal hypersurfaces in the product space 
$M\times\R$, where $M$ is a Cartan-Hadamard manifold (see for example, \cite{DHL}, \cite{ER}, \cite{MR}, 
\cite{NR}, \cite{RT}, \cite{RSS}, \cite{Spruck}). 

Very recently in \cite{CHR}, the authors generalized (most of) the solvability results to a larger class 
of operators $\mathcal{Q}$ of the form
\begin{equation}
\label{defopQ}
\mathcal{Q}[u]= \diver (\cP (|\nabla u|^2) \nabla u),
\end{equation}
with $\cP$ subject to the following growth conditions. Let ${\mathcal{P}}\colon (0,\infty)\to[0,\infty)$ 
be a smooth function such that
\begin{equation}
\label{Agrowth}{\mathcal{P}}(t)\le P_{0}t^{(p-2)/2}%
\end{equation}
for all $t>0$, with some constants $P_{0}>0$ and $p\ge1$, and that
${\mathcal{B}}:={\mathcal{P}}^{\prime}/{\mathcal{P}}$ satisfies
\begin{equation}
\label{Bgrowth}-\frac{1}{2t} < {\mathcal{B}}(t)\le\frac{B_{0}}{t}%
\end{equation}
for all $t>0$ with some constant $B_{0}>-1/2$. Furthermore, assume that
$t{\mathcal{P}}(t^{2})\to0$ as $t\to 0+$ and define ${\mathcal{P}%
}(\lvert X \rvert^{2})X=0$ whenever $X$ is a zero vector. 

Following \cite{CHR} we call a relatively compact open set $\Omega\Subset M$ \emph{${\mathcal{Q}}%
$-regular\/} if for any continuous boundary data $h\in C(\partial\Omega)$ there
exists a unique $u\in C(\bar{\Omega})$ which is ${\mathcal{Q}}$-solution in
$\Omega$ and $u|\partial\Omega=h$. In addition to the growth conditions on
${\mathcal{P}}$, assume that

\begin{itemize}
\item[(A)] there is an exhaustion of $M$ by an increasing sequence of
${\mathcal{Q}}$-regular domains $\Omega_{k}$, and that

\item[(B)] locally uniformly bounded sequences of continuous ${\mathcal{Q}}%
$-solutions are compact in relatively compact subsets of $M$.
\end{itemize}
It is well-known that the conditions above are satisfied by the minimal graph operator and the $p$-
Laplacian (see \cite{DLR}, \cite{HKM} and \cite{Spruck}). 

The main theorem in \cite{CHR} is a solvability result for the
asymptotic Dirichlet problem for operators ${\mathcal{Q}}$ that satisfy
\eqref{Agrowth}, \eqref{Bgrowth}, and conditions (A) and (B) under curvature
assumption
\[
-b\bigl(r(x)\bigr)^{2}\leq K(P)\leq-a\bigl(r(x)\bigr)^{2}%
\]
on $M$, where $P\subset T_x M$ is a $2$-plane and 
$a,b\colon \lbrack0,\infty)\to\lbrack0,\infty),\ b\geq a,$ are smooth functions 
satisfying suitable assumptions. Here, instead of giving the precise assumptions on functions $a$ and $b$, 
we state the following two solvability results as special cases of their main theorem (Theorem 1.6 in \cite{CHR}).
\begin{theorem}\cite[Theorem 1.5]{CHR}
\label{CHR_thm1} Let $M$ be a Cartan-Hadamard manifold of dimension $n\ge2$. Fix
$o\in M$ and set $r(\cdot)=d(o,\cdot)$, where $d$ is the Riemannian
distance in $M$. Assume that
\begin{equation*}
-r(x)^{2\left(  \phi-2\right)  -\varepsilon}%
\leq\sect_{x}(P)\leq-\dfrac{\phi(\phi-1)}{r(x)^{2}},
\end{equation*}
for some constants $\phi>1$ and $\varepsilon>0,$ where $\sect_{x}(P)$ is the
sectional curvature of a plane $P\subset T_{x}M$ and $x$ is any point in the
complement of a ball $B(o,R_{0})$. Then the asymptotic Dirichlet problem for
the minimal graph equation \eqref{Mequ} is uniquely solvable for any
boundary data $f\in C\bigl(M(\infty)\bigr)$.
\end{theorem}
\begin{theorem}\cite[Corollary 1.7]{CHR}
\label{HVkor2_RT} Let $M$ be a Cartan-Hadamard manifold of dimension $n\ge2$.
Fix $o\in M$ and set $r(\cdot)=d(o,\cdot)$, where $d$ is the Riemannian
distance in $M$. Assume that
\begin{equation}\label{curv_assump_k}
-r(x)^{-2-\varepsilon}e^{2kr(x)}\le\sect_{x}%
(P)\le-k^{2}%
\end{equation}
for some constants $k>0$ and $\varepsilon>0$ and for all $x\in M\setminus
B(o,R_{0})$. Then the asymptotic Dirichlet problem for the operator $\mathcal{Q}$ (defined as in \eqref{defopQ})
is uniquely solvable for any boundary data $f\in C\bigl(M(\infty)\bigr)$.
\end{theorem}

The Dirichlet problem at infinity for $\cA$-harmonic functions, that is solutions of
\begin{equation}\label{Aequ}
\mathcal{Q}[u]= -\diver \cA_x(\nabla u)=0,
\end{equation}
 has been considered in \cite{Va1} and \cite{Va2}. In 
\cite{Va2},  V\"ah\"akangas was able to generalize the result obtained in \cite{HoVa} (for the $p$-Laplacian) to the 
$\cA$-harmonic case. 
In \cite{Va1}, by generalizing a method due to Cheng~\cite{cheng}, he solved the asymptotic Dirichlet problem for $
\cA$-harmonic equations of type $p$ provided the radial sectional curvatures outside a compact set satisfy 
\[
K(P) \le-\frac{\phi(\phi-1)}{r^{2}(x)}
\]
for some constant $\phi>1$ with $1<p<1+\phi(n-1)$ and 
\[
\lvert K(P) \rvert\le C\lvert K(P^{\prime}) \rvert
\]
for some constant $C$, where $P$ and $P^{\prime}$
are any $2$-dimensional subspaces of $T_{x} M$ containing the (radial) vector
$\nabla r(x)$. It is worth observing that no curvature lower bounds are
needed here. In the recent preprint \cite{CHH1}, Casteras, Heinonen, and Holopainen generalized this result for the minimal graph equation.

The goal of this paper is threefold. First of all, we are looking for an optimal (or at least close 
to optimal) curvature upper bound under which asymptotic Dirichlet problems for equations \eqref{Mequ} 
and \eqref{Aequ} are solvable provided an appropriate curvature lower bound holds. Secondly, we are using 
PDE methods, like Caccioppoli-type inequalities (Lemma~\ref{newlem-mincaccio}), Moser iteration scheme 
(Lemma~\ref{moseritemin}), and Young complementary functions 
to study the minimal graph 
equation. As far as we know such methods are not frequently used in the context of the minimal graph 
equation. Last but not least, we want to publicize the results and methods of the still unpublished 
preprint \cite{Va2} of  V\"ah\"akangas.
Our main results are the following two theorems. Below in 
Theorem~\ref{thmharm}, the operator $\cA$ is of type 
$p\in (1,\infty)$ and $\alpha$ and $\beta$ are so-called structural constants $\cA$; see Section 
\ref{subsec_A-harm} for details.

\begin{theorem}\label{thmharm}
Let $M$ be a Cartan-Hadamard manifold of dimension $n\ge 2$. 
Assume that 
\begin{equation}\label{curv_ass_minim2}
-\frac{\bigl(\log r(x)\bigr)^{2\bar{\varepsilon}}}{r(x)^2} \leq K(P)\leq-\frac
{1+\varepsilon}{r(x)^{2}\log r(x)},
\end{equation}
for some constants $\varepsilon>\bar{\varepsilon}>0,$ where $K(P)$ is the sectional
curvature of any plane $P\subset T_{x}M$ that contains the radial vector $\nabla r(x)$ and $x$ is any point in the complement of a ball $B(o,R_0)$.
Then the asymptotic Dirichlet problem for the $\cA$-harmonic equation \eqref{Aequ} is
uniquely solvable for any boundary data $f\in C\bigl(\partial_\infty M\bigr)$ provided that 
$1< p< n\alpha/\beta$.
\end{theorem}

\begin{theorem}\label{thmmin}
Let $M$ be a Cartan-Hadamard manifold of dimension $n\ge 3$ satisfying the curvature assumption
\eqref{curv_ass_minim2} for all $2$-planes $P\subset T_{x}M$, with $x\in M\setminus B(o,R_0)$. 
Then the asymptotic Dirichlet problem for the minimal graph equation \eqref{Mequ} is
uniquely solvable for any boundary data $f\in C\bigl(\partial_\infty M\bigr)$.
\end{theorem}

We notice that the Laplace-Beltrami operator corresponds to the case $p=2$ and $\alpha=\beta=1$, and therefore is 
covered by Theorem~\ref{thmharm} in dimensions $n\ge 3$. Thus we obtain a generalization to higher dimensions of a 
recent result by Neel \cite{Neel}. 
\begin{corollary}\label{coromain}
Let $M$ be a Cartan-Hadamard manifold of dimension $n\ge 3$ satisfying the curvature assumption
\eqref{curv_ass_minim2} for all $2$-planes $P\subset T_{x}M$ that contain the radial vector $\nabla r(x)$,
with $x\in M\setminus B(o,R_0)$.
Then the asymptotic Dirichlet problem for the Laplace-Beltrami equation is uniquely solvable for any boundary 
data $f\in C\bigl(\partial_\infty M\bigr)$.
\end{corollary}
The curvature upper bound \eqref{curv_ass_minim2} appears also in a recent paper \cite{RTgeomdedi} where Ripoll and 
Telichevesky solved the asymptotic Dirichlet problem for the minimal graph equation on rotationally symmetric 
Hadamard surfaces. Notice that dimension $n=2$ is excluded in Theorem~\ref{thmmin}. However, we believe that the result holds also in the 2-dimensional setting.

We point out that our curvature assumptions, in particular the upper bounds, are in a sense optimal. 
Indeed, in \cite{march} March characterized the existence of nonconstant bounded harmonic functions on 
rotationally symmetric Riemannian manifolds $M=(\R^n,g)$, where the Riemannian metric $g$ is given in 
polar coordinates as
\[
ds^2=dr^2+f(r)^2d\theta^2.
\]
He proved that $M$ carries nonconstant bounded harmonic functions if and only if
\[
I(f)=\int_1^\infty\left( f(s)^{n-3}\int_s^\infty f(t)^{1-n}dt\right)ds<\infty.
\]
Letting $c_2=1$ and $c_n=1/2$ for $n\ge 3$ and assuming that radial sectional curvatures 
$K_r=-(f''/f)(r)$ are nonpositive and have upper bound
\[
K_r\le-\frac{c}{r^2\log r}
\]
for some constant $c>c_n$ and for all large $r$ we have $I(f)<\infty$. On the other hand, if
\[
K_r\ge -\frac{c}{r^2\log r}
\]
for $c<c_n$ and for all large $r$, then $I(f)=\infty$. Since all bounded harmonic functions on such $M$ 
are constant, the asymptotic Dirichlet problem for the Laplace-Beltrami operator can not be solvable. 
In general, assume 
that 
\[
K(P_x)\geq-\frac{1}{r(x)^{2}\log r(x)}
\]
for all large $r(x)$ and let us consider first the case of an $\cA$-harmonic operator of type $p\ge n$. 
The standard Bishop-Gromov volume comparison theorem for geodesic balls $B_r=B(o,r)$ gives
\[
\vol(\partial B_r)\leq C (r\log r)^{n-1}
\]
for some constant $C$ and for all $r\ge r_0$ large enough. It follows that
\[
\int_{r_0}^\infty \frac{dr}{(\vol(\partial B_r)^{1/(p-1)}}\ge 
C\int_{r_0}^\infty \frac{dr}{r\log r}=\infty
\]
which implies that $M$ is so-called $p$-parabolic and hence every bounded $\cA$-harmonic function (with 
$\cA$ of type $p$) is constant; see e.g. \cite{holDuke} and \cite{CoHoSa-Co}. 
On the other hand, in \cite{RiSe}  Rigoli and Setti proved the following nonexistence theorem:
\begin{theorem}
Let $M$ be a complete manifold and $u\in C^1(M)$ be a solution of
\[
\diver \frac{\varphi (|\nabla u|)\nabla u}{|\nabla u|}=0,
\]
where $\varphi \in C^1 ((0,\infty))\cap C^0([0,\infty))$ satisfies the following conditions:
\begin{enumerate}
\item $\varphi (0)=0$, 
\item $\varphi (t)>0$, for all $t\geq 0$,
\item $\varphi(t)\leq A t^\delta$, for all $t\geq 0$,
\end{enumerate}
for some positive constants $A$ and $\delta$. Assume that
\[
(\vol(\partial B_r)^{\frac{1}{\delta}})^{-1}\notin L^1(\infty),
\]
then $M$ is $\varphi$-parabolic i.e. $u$ is constant.
\end{theorem}

Using this theorem, we also see that the curvature upper bound would be sharp for the minimal graph equation in dimension $n=2$. Notice that $\delta=1$ for the minimal graph equation.
We close this introduction with some comments on the necessity of curvature lower bounds. 
Indeed, Ancona's and Borb\'ely's examples (\cite{ancrevista}, \cite{Bor}) show that a (strictly)
negative curvature upper bound alone is not sufficient for the solvability of
the asymptotic Dirichlet problem for the Laplace equation. In \cite{H_ns},
Holopainen generalized Borb\'ely's result to the $p$-Laplace equation, and
very recently, Holopainen and Ripoll \cite{HR_ns} extended these
nonsolvability results to the operator $\mathcal{Q}$ (as defined in \eqref{defopQ}), in particular, to the
minimal graph equation.

The plan of the paper is the following: Section~\ref{sec_preli} is devoted to preliminaries. We recall 
some well-known facts on Cartan-Hadamard manifolds, Jacobi equations, $\cA$-harmonic functions, the minimal 
graph equation, and Young 
functions. In Section~\ref{sec_asdir} we prove Theorem \ref{thmharm}. We adopt the same strategy as the 
one used in \cite{Va2}. It is based on a Moser iteration procedure involving a weighted Poincar\'e 
inequality. 
Finally, in Section~\ref{asdirmin} we prove Theorem \ref{thmmin} 
adapting to the minimal graph equation the method used in Section~\ref{sec_asdir} for $\cA$-harmonic 
functions. In this case since this equation does not satisfy \eqref{aharmstruc}, some extra 
difficulties appear.  
\subsubsection*{Acknowledgement}
We would like to thank Joel Spruck for his help to obtain the decay estimate for $|\nabla\log W|$ in 
Lemma~\ref{W-decay}.

\section{Preliminaries}\label{sec_preli}
\subsection{Cartan-Hadamard manifolds}\label{subsec_hada}
We recall that Cartan-Hadamard manifolds are complete simply connected Riemannian
$n$-manifolds, $n\geq 2$, with nonpositive sectional curvature. Let $M$ be a Cartan-Hadamard manifold, 
$\partial_\infty M$ the sphere at infinity, and $\bar M=M\cup \partial_\infty M$. Recall that the sphere
at infinity is defined as the set of all equivalence classes of unit speed
geodesic rays in $M$; two such rays $\gamma_{1}$ and $\gamma_{2}$ are
equivalent if $\sup_{t\ge0}d\bigl(\gamma_{1}(t),\gamma_{2}(t)\bigr)< \infty$.
For each $x\in M$ and $y\in\bar M\setminus\{x\}$ there exists a unique unit
speed geodesic $\gamma^{x,y}\colon \mathbb{R}\to M$ such that $\gamma^{x,y}_{0}=x$ and 
$\gamma^{x,y}_{t}=y$ for some $t\in(0,\infty]$. If $v\in
T_{x}M\setminus\{0\}$, $\alpha>0$, and $R>0$, we define a cone
\[
C(v,\alpha)=\{y\in\bar M\setminus\{x\}\colon \sphericalangle(v,\dot\gamma^{x,y}%
_{0})<\alpha\}
\]
and a truncated cone
\[
T(v,\alpha,R)=C(v,\alpha)\setminus\bar B(x,R),
\]
where $\sphericalangle(v,\dot\gamma^{x,y}_{0})$ is the angle between vectors
$v$ and $\dot\gamma^{x,y}_{0}$ in $T_{x} M$. All truncated cones and open balls in $M$
form a basis for the cone topology on $\bar M$.

\subsection{Jacobi equation}\label{subsec_jac}
We use the curvature upper bound in order to prove a weighted Poincar\'e inequality and to estimate from above the 
norm of the gradient of an angular function. The curvature lower bound, in turn, is used to estimate the volume form 
from above. All of these estimates will be given in terms of solutions to a $1$-dimensional Jacobi equation. 
If $k\colon [0,\infty)\to[0,\infty)$ is a smooth function, we
denote by $f_{k}\in C^{\infty}\bigl([0,\infty)\bigr)$ the solution to the
initial value problem
\begin{equation}
\label{Jacobi_eq}
\begin{cases} f_k(0)=0, \\ f_k'(0)=1, \\ f_k''=k^2 f_k. \end{cases} 
\end{equation}
It follows that the solution $f_{k}$ is a nonnegative smooth function. Concerning the curvature upper bound
in \eqref{curv_ass_minim2} we have the following estimates:
\begin{proposition}\cite[Prop. 3.4]{choi}
\label{jacest}
Suppose that $f\colon [R_0,\infty) \to \R$, $R_0>0$, is a positive strictly increasing function satisfying the 
equation $f^{\prime \prime}(r)=a^2(r) f(r)$, where 
\[
a^2(r)\geq \frac{1+\varepsilon}{r^2 \log r}, 
\]
for some $\varepsilon>0$ on $[R_0,\infty)$. Then, for any $0<\tilde{\varepsilon}<\varepsilon$, 
there exists $R_1\geq R_0$ such that, for all $r\geq R_1$,
\[
f(r)\geq r (\log r)^{1+\tilde{\varepsilon}},\quad 
\frac{f^\prime (r)}{f(r)}\geq \frac{1}{r}+\frac{1+\tilde{\varepsilon}}{r\log r}.
\]
\end{proposition}

\subsection{$\cA$-harmonic functions and Perron's method}\label{subsec_A-harm}
In this section we define $\cA$-harmonic and $\cA$-superharmonic functions and record their basic properties that 
will be relevant in the sequel. We refer to \cite{HKM} for the proofs and for the nonlinear potential theory 
of $\cA$-harmonic and $\cA$-superharmonic functions. 

Let $\Omega$ be an open subset of a Riemannian manifold $M$. Suppose that for a.e. $x\in \Omega$ 
we are given a continuous map $\cA_{x}\colon T_{x}M\to T_{x}M$ such that the map $x\mapsto\cA_{x}(X_x)$ 
is a measurable vector field whenever $X$ is.
We assume further that there are constants $1<p<\infty$ and $0<\alpha\le\beta<\infty$ such that 
for a.e. $x\in \Omega$, for all $v,w\in T_xM,\ v\ne w$, and for all
$\lambda\in\R\setminus\{0\}$ we have
\begin{equation}\label{aharmstruc}
\begin{split}
\langle\cA_x(v),v\rangle & \ge\alpha|v|^p;\\
|\cA_x(v)|               & \le\beta|v|^{p-1};\\
\langle \cA_x (v) - &\cA_x (w) ,  v-w\rangle >0;\\
\cA_x (\lambda v) &=\lambda |\lambda|^{p-2} \cA_x (v).
\end{split}
\end{equation}
We denote the set of such operators by $\cA^p (\Omega)$ and we say that $\cA$ is of type $p$. 
The constants $\alpha$ and $\beta$ are called the structure constants of $\cA$. 

A function $u\in W^{1,p}_{\text{loc}}(\Omega)$ is called a (weak) solution of the equation
\begin{equation}\label{eq_waharm}
-\diver\cA_{x}(\nabla u)=0
\end{equation}
in $\Omega$ if 
\begin{equation}\label{eq_aharm}
\int_{\Omega}\langle\cA_{x}(\nabla u),\nabla\varphi\rangle=0
\end{equation}
for all $\varphi\in C^{\infty}_{0}(\Omega).$ If $\abs{\nabla u}\in L^p(\Omega)$, it is equivalent to require
\eqref{eq_aharm} for all $\varphi\in W^{1,p}_0(\Omega)$ by approximation.  
Continuous solutions of \eqref{eq_waharm} are called 
\define{$\cA$-harmonic functions} (of type $p$). By the fundamental work of Serrin \cite{serrin}, 
every solution of \eqref{eq_waharm} has a continuous representative. In the special case 
$\cA_{x}(h)=\abs{h}^{p-2}h,$ $\cA$-harmonic functions are called \define{$p$-harmonic} and, in particular, 
if $p=2,$ we obtain the usual harmonic functions. 

A function $u\in W^{1,p}_{\text{loc}}(\Omega)$ is a \define{subsolution} of 
\eqref{eq_waharm} in $\Omega$ if
\[
-\diver\cA_{x}(\nabla u)\le 0
\]
weakly in $\Omega,$ that is
\[
\int_{\Omega}\ang{\cA_{x}(\nabla u),\nabla\varphi}\le 0
\]
for all nonnegative $\varphi\in C^{\infty}_{0}(\Omega).$ A function $v$ is
called \define{supersolution} of \eqref{eq_waharm} if $-v$ is a subsolution.
Finally, a lower semicontinuous function $u\colon \Omega\to (-\infty,+\infty]$ that is not identically 
$+\infty$ in any component of $\Omega$ is called \define{$\cA$-superharmonic} if for every open $D\Subset \Omega$ 
and for every $h\in C(\bar D)$ that is $\cA$-harmonic in $D$, $h\le u$ on $\partial D$ implies $h\le u$ 
in $D$.

A fundamental feature of (sub/super)solutions of \eqref{eq_waharm} is the following well-known 
\define{comparison principle}: If $u\in W^{1,p}(\Omega)$ is a supersolution and $v\in W^{1,p}$ a subsolution 
of \eqref{eq_waharm} in $\Omega$ such that $\max(v-u,0)\in W^{1,p}_0(\Omega)$, then $u\ge v$ a.e. in $\Omega$. The 
existence of $\cA$-harmonic functions is given by the following result. Suppose that $\Omega\Subset M$ is a 
relatively compact (nonempty) open set and that
$\theta\in W^{1,p}(\Omega)$. Then there exists a unique $\cA$-harmonic function $u$ in $\Omega$, with 
$u-\theta\in  W^{1,p}_0(\Omega)$.
 
Given  a function $f \in C(\partial_\infty M)$ the Dirichlet problem at infinity for $\cA$-harmonic 
functions consists in finding a function $u\in C(\bar{M})$ such that $\cA(u)=0$ in $M$ and 
$u|_{\partial_\infty M}= f.$
In order to solve the Dirichlet problem for the $\cA$-harmonic functions, we will use Perron's method. 
Let $\cA \in \cA^p (M)$, with $p\in (1,\infty)$. We begin by recalling the definition of the upper class of a function $f\in \partial_\infty M$.
\begin{definition}
A function $u\colon M\to (-\infty ,\infty]$ belongs to the upper class $\cU_f$ of $f\colon \partial_\infty M\to [-\infty,\infty]$ if
\begin{enumerate}
\item $u$ is $\cA$-superharmonic in $M$,
\item $u$ is bounded from below, and
\item $\displaystyle\liminf_{x\to x_0} u(x) \geq f(x_0)$, for all $x_0\in \partial_\infty M$.
\end{enumerate}
\end{definition} 
The function
$$\overline{H}_f=\inf \{ u\colon u\in \cU_f\}$$
is called the upper Perron solution.
\begin{theorem}
One of the following is true:
\begin{enumerate}
\item $\overline{H}_f$ is $\cA$-harmonic in $M$,
\item $\overline{H}_f\equiv \infty$ in $M$,
\item $\overline{H}_f\equiv -\infty$ in $M$.
\end{enumerate}
\end{theorem}
Next we define $\cA$-regular points at infinity.
\begin{definition}
A point $x_0\in \partial_\infty M$ is called $\cA$-regular if
\[
\lim_{ \stackrel{ \mbox{\scriptsize$x\to x_0$} }{x\in M}} \overline{H}_f(x)=f(x_0)
\]
for all $f\in C(\partial_\infty M)$.
\end{definition}
It is easy to see that the Dirichlet problem at infinity for $\cA$-harmonic functions is uniquely solvable
if every point at infinity is $\cA$-regular.

\subsection{Minimal graph equation}\label{subsec-minimal}
Let $\Omega\subset M$ be an open set. We say that a function $u\in W^{1,1}_{\loc}(\Omega)$ is a 
\define{(weak) solution of the minimal graph equation} \eqref{Mequ} if 
\begin{equation}\label{weak-Mequ}
\int_{\Omega}\frac{\langle\nabla u,\nabla\varphi\rangle}{\sqrt{1+|\nabla u|^2}}=0
\end{equation}
for every $\varphi\in C^{\infty}_{0}(\Omega)$. Note that the integral above is well-defined since
\[
\sqrt{1+|\nabla u|^2}\ge |\nabla u|\quad \text{a.e.}, 
\]
and therefore
\[
\int_{\Omega}\frac{\bigl|\langle\nabla u,\nabla\varphi\rangle\bigr|}{\sqrt{1+|\nabla u|^2}}
\le \int_{\Omega}\frac{|\nabla u| |\nabla\varphi|}{\sqrt{1+|\nabla u|^2}} \le 
\int_{\Omega}|\nabla\varphi|<\infty.
\]
In fact, it is equivalent to require \eqref{weak-Mequ} for every
$\varphi\in \mathring{W}^{1,1}_{0}(\Omega)$. Indeed, let $\varphi\in \mathring{W}^{1,1}_{0}(\Omega)$ and let 
$(\varphi_j)$ be a sequence in $C^{\infty}_0(\Omega)$ such that
$\nabla\varphi_j\to\nabla\varphi$ in $L^1(\Omega)$. Supposing that \eqref{weak-Mequ} holds for all such $\varphi_j$, 
we get
\begin{align*}
\left|\int_{\Omega}\frac{\langle\nabla u,\nabla\varphi\rangle}{\sqrt{1+|\nabla u|^2}}\right| &=
\left|\int_{\Omega}\frac{\langle\nabla u,\nabla\varphi\rangle}{\sqrt{1+|\nabla u|^2}} -
\int_{\Omega}\frac{\langle\nabla u,\nabla\varphi_j\rangle}{\sqrt{1+|\nabla u|^2}}\right|\\
&=\left|\int_{\Omega}\frac{\langle\nabla u,\nabla\varphi-\nabla\varphi_j\rangle}{\sqrt{1+|\nabla u|^2}}\right|\\
&\le \int_{\Omega}\frac{|\nabla u||\nabla\varphi-\nabla\varphi_j|}{\sqrt{1+|\nabla u|^2}}
\le\int_{\Omega}|\nabla\varphi-\nabla\varphi_j|\to 0
\end{align*}
as $j\to 0$.
The following lemma guarantees the existence of (strong) solutions of \eqref{Mequ} with given boundary values.
\begin{lemma}\label{Mequ-sol-exist}
Suppose that $\Omega\Subset M$ is a smooth relatively compact open set whose boundary has nonnegative mean curvature 
with respect to inwards pointing unit normal field. Then for each $f\in C^{2,\alpha}(\bar{\Omega})$ there exists a 
unique $u\in C^{\infty}(\Omega)\cap C^{2,\alpha}(\bar{\Omega})$ that solves the minimal graph equation \eqref{Mequ}
in $\Omega$ with boundary values $u|\partial\Omega=f|\partial\Omega$. 
\end{lemma}
\begin{proof}
This lemma follows from well-known techniques used in the continuity method of elliptic PDE theory and therefore we 
just sketch the argument. Set
\[
V=\{t\in [0,1]\colon \exists u\in C^{2,\alpha}(\bar{\Omega})\text{ such that $\cM[u]=0$ in $\Omega$ and
 $u|\partial\Omega=tf|\partial\Omega$}\}.
\]
We have $V\neq\emptyset$ since $0\in V.$ Moreover, by the implicit function
theorem, $V$ is open in $[0,1]$. Given $t\in V$, let $u$ be a solution of \eqref{Mequ} such that 
$u|\partial\Omega=tf|\partial\Omega$. Since constant functions are solutions of \eqref{Mequ}, we 
have $\sup_{\bar{\Omega}}|u|\le\max_{\partial\Omega}|f|$ by the comparison
principle (see e.g. \cite[Lemma 1]{CHR}. Also, since $\partial\Omega$ has nonnegative mean curvature with 
respect to  inwards pointing unit normal field, we may use classical logarithmic type barriers to prove that 
$\max_{\partial\Omega}|\nabla u| \leq C$ where $C$ is a constant that depends only on $f$ and on
$\Omega$ (see e.g. \cite[Section 4]{DHL} for details). By \cite[Lemma 3.1]{RSS} we have
$\max_{\bar{\Omega}} |\nabla u | \leq C$ for some constant independent of $u$ and $t$. H\"{o}lder 
estimates and theory of linear elliptic PDEs imply that the $C^{2,\beta}$ norm of $u$ is bounded by a 
constant depending only on $f$ and $\Omega$ for some $0<\beta<1.$ Then, if $t_{n}\in V$
converges to $t\in [ 0,1] $ and $u_{n}$ is a solution of \eqref{Mequ} such
that $u_{n}|\partial\Omega=t_{n}f|\partial\Omega$, then $(u_{n})$ contains a subsequence converging in 
the $C^{2}$ norm on $\bar{\Omega}$ to a solution $u\in C^{2}(\bar{\Omega})$ of \eqref{Mequ} in 
$\Omega$ such that $u|\partial\Omega=tf|\partial\Omega$. Regularity theory implies that
$u\in C^{\infty}(\Omega)\cap C^{2,\alpha}(\bar{\Omega})$. It follows that $t\in V$, so that
$V$ is closed and hence $V=[0,1]$.
\end{proof}

From now on we will mainly consider solutions of \eqref{Mequ} that are at least $C^2$-smooth.

\subsection{Young functions}\label{subsec_young}
Let $\phi\colon [0,\infty)\to [0,\infty)$ be a homeomorphism and let $\psi=\phi^{-1}$. Define
\define{Young functions} $\Phi$ and $\Psi$ by setting, for each $t\in [0,\infty)$,
\begin{align*}
\Phi(t)=\int_{0}^t\phi(s)ds\quad
\text{and}\quad
\Psi(t)=\int_{0}^t\psi(s)ds.
\end{align*}
Then we have the \define{Young inequality}
\[
ab\le \Phi(a) +\Psi(b)
\]
for all $a,b\in [0,\infty)$. The functions $\Phi$ and $\Psi$ are said to form a 
\define{complementary Young pair}. Furthermore, $\Phi$ (and similarly $\Psi$) is a continuous, strictly 
increasing, and convex function satisfying
\begin{align*}
\lim_{t\to 0+}\frac{\Phi(t)}{t}=0\quad
\text{and}\quad
\lim_{t\to\infty}\frac{\Phi(t)}{t}=\infty.
\end{align*}
Such Young functions are usually called $N$-functions (nice Young functions) in the literature;
see e.g. \cite{Kuf} for a more general definition of Young functions.

Following \cite{Va2} we consider complementary Young pairs of a special type. Suppose that a 
homeomorphism $G\colon [0,\infty)\to[0,\infty )$ is a Young function that is a diffeomorphism on 
$(0,\infty)$ and satisfies
\begin{equation}
\label{youngA1}
\int_0^1 \dfrac{1}{G^{-1}(t)}dt<\infty,
\end{equation} 
and
\begin{equation}
\label{youngA2}
\lim_{t\to 0} \dfrac{t G^\prime (t)}{G(t)}=1.
\end{equation}
Then $G(\cdot^{1/p})^p$, with $p\geq 1$, is also a Young function and we can define  
$F\colon [0,\infty)\to [0,\infty)$ so that $G(\cdot^{1/p})^p$ and $F(\cdot^{1/p})$ form a complementary 
Young pair. The space of such functions $F$ will be denoted by $\cF_p$. Note that if $F\in\cF_p$, then 
$\lambda F\in\cF_p$ and $F(\lambda\cdot)\in\cF_p$ for every positive constant $\lambda$.
It is proved in \cite{Va2} that $\cF_p$ is non-empty. More precisely, we have the following:
\begin{proposition}\cite[Proposition 4.3]{Va2}
\label{growthyoung}
Fix $\varepsilon_0 \in (0,1)$. There exists $F\in \cF_p$ such that
\begin{equation}\label{eq-F-upper}
F(t)\leq 
t^{p+\varepsilon_0}\exp\left(-\tfrac{1}{t}\left(\log\bigl(e+\tfrac{1}{t}\bigr)\right)^{-1-\varepsilon_0}\right) 
\end{equation}
for all $t\in [0,\infty)$.
\end{proposition}
We omit the details of the proof of Proposition~\ref{growthyoung} and refer to \cite{Va2}; see also \cite{CHH1}.  Here we just 
sketch the construction. The function $F$ is obtained by first choosing $\lambda\in (1,1+\varepsilon_0)$ 
and a homeomorphism $H\colon [0,\infty)\to [0,\infty)$ that is diffeomorphic on $(0,\infty)$ and satisfies
\[
H(t)=\begin{cases}
\left(\log\tfrac 1t\right)^{-1}\left(\log\log\tfrac 1t\right)^{-\lambda} 
&\mbox{if $t$ is small enough};\\
t^{p/\varepsilon_0} &\mbox{if $t$ is large enough},
\end{cases}
\]  
and then setting $G(t)=\int_0^t H(s)ds$. Then $G$ and $\tilde{G},\ \tilde{G}(t)=G(t^{1/p})^p$, are 
Young functions. Let $\tilde{F}$ be the complementary Young function to $\tilde{G}$ and, finally,
define $F$ by setting $F(t)=c\tilde{F}(t^p)$ for a suitable positive constant $c$.

Since $G$ is convex, we have $G(t)\ge ct$ for all $t\ge 1$. Therefore $G^{-1}(t)\le ct$ for 
all $t$ large enough and, consequently, $\int_{0}^{\infty}1/G^{-1}=\infty$. Taking into account 
\eqref{youngA1} we conclude that the function 
$\gamma$, defined by
\[
\gamma(t)=\int_{0}^{t}\frac{1}{G^{-1}(s)}ds,
\]
is a homeomorphism $[0,\infty)\to [0,\infty)$ that is a diffeomorphism on $(0,\infty)$.
Hence the same is true for its inverse 
\begin{equation}\label{eq-defvarphi}
\varphi=\gamma^{-1}\colon [0,\infty)\to [0,\infty).
\end{equation}
We collect the properties of such a function $\varphi$ into the following lemma.
\begin{lemma}\cite[Lemma 4.5]{Va2}
The function $\varphi\colon [0,\infty)\to [0,\infty)$ is a homeomorphism that is smooth on 
$(0,\infty)$ and satisfies
\begin{equation}
\label{relvarphi1}
G \circ \varphi^\prime =\varphi,
\end{equation}
and
\begin{equation}
\label{relvarphi}
\lim_{t\to 0+} \dfrac{\varphi^{\prime \prime}(t) \varphi (t)}{\varphi^\prime (t)^2}=1.
\end{equation}
\end{lemma} 

From now on, $\varphi$ will be the function defined in \eqref{eq-defvarphi} such that the corresponding
$F\in\cF_p$ satisfies \eqref{eq-F-upper}. We define an auxiliary 
function $\psi=(\varphi^\prime)^{p-1}\varphi$. It is easy to see that 
$\psi\colon [0,\infty)\to [0,\infty)$ is a homeomorphism that is smooth on $(0,\infty)$. 
It follows from \eqref{relvarphi} that 
\begin{equation}
\label{relpsi}
\lim_{t\to 0+}\dfrac{\psi^\prime (t)}{\varphi^\prime (t)^p}=p.
\end{equation}
Consequently, for every $\delta>0$, there exists $t_{\delta}>0$ such that 
\begin{equation}\label{(2.18)}
\frac{\psi'(t)}{2p}\le \varphi'(t)^p\le \frac{(1+\delta)^p\psi'(t)}p
\end{equation}
and
\begin{equation}\label{(2.19)}
\frac{\psi(t)^p}{\psi'(t)^{p-1}}\le\frac{(1+\delta)^p\varphi(t)^p}{p^{p-1}}
\end{equation}
whenever $t\in (0,t_{\delta}]$.
\section{Dirichlet problem at infinity for $\cA$-harmonic functions}\label{sec_asdir}
This section is devoted to the proof of Theorem~\ref{thmharm}. We assume that $\cA\in\cA_p(M)$, where $1<p<\infty$, and that 
$\alpha$ and $\beta$ are the structural constants of $\cA$ as in \ref{subsec_A-harm}. Throughout the section the function $F\in\cF_p$ satisfies \eqref{eq-F-upper} 
(see Proposition~\ref{growthyoung}) and the function $\varphi$ is related to $G$ and $F$ by \eqref{relvarphi1}
as explained in \ref{subsec_young}. Furthermore, $r$ stands for the distance function $r(x)=d(x,o)$.

We start with stating a Caccioppoli-type inequality that will be crucial in the sequel.   
\begin{lemma}\cite[Lemma 2.15]{Va2}
\label{caccio}
Suppose that $\Psi\colon [0,\infty)\to [0,\infty)$ is a homeomorphism that is smooth on $(0,\infty)$.
Let $U\Subset M$ be an open, relatively compact set and let $\eta \geq 0$ be a 
Lipschitz function in $U$. 
Suppose that $\theta,u\in L^\infty (U)\cap W^{1,p}(U)$ are continuous functions and that $u$ is 
$\cA$-harmonic in $U$. Denote $h=|u-\theta|$ and suppose that
\[
\eta^p \Psi (h) \in W^{1,p}_0 (U).
\]
Then
\begin{align}\label{ineq-caccio}
\left(\int_U \eta^p \Psi^\prime (h) |\nabla u|^p\right)^{1/p} 
&\leq  \frac{\beta}{\alpha} \left(\int_U \eta^p \Psi^\prime (h) |\nabla \theta|^p\right)^{1/p}\\
\notag &\quad +\frac{p\beta}{\alpha}\left(\int_U \frac{\Psi^p}{(\Psi^\prime)^{p-1}}(h) |\nabla \eta|^p\right)^{1/p}.
\end{align}
\end{lemma}
The proof is a straightforward application of the $\cA$-harmonic equation \eqref{eq_aharm} for $u$ with 
the test function
$f=\eta^p \Psi ((u-\theta)^+) -\eta^p \Psi ((u-\theta)^-)$.
We omit the details and refer to \cite{Va2} for the proof. In Section~\ref{asdirmin} we prove a  
Caccioppoli inequality for solutions of the minimal graph equation.

Combining the Caccioppoli inequality \eqref{ineq-caccio} with a local Sobolev inequality (see 
\eqref{ineq-sobo} below) and running a Moser-type iteration we obtain pointwise estimates for the 
difference of an $\cA$-harmonic function and its boundary data in sufficiently small balls in terms 
of certain integral quantities in bigger balls. Recall that a local Sobolev inequality holds on any 
Cartan-Hadamard manifold. 
More precisely, there exist two constants $r_S>0$ and $C_S<\infty$ such that
\begin{equation}\label{ineq-sobo}
\left(\int_{B} |\eta|^{n/(n-1)} \right)^{(n-1)/n}\leq C_S \int_{B}|\nabla \eta|
\end{equation}
for every ball $B=B(x,r_S)\subset M$ and every function $\eta \in C_0^\infty (B)$.
Such an inequality was obtained by Hoffman and Spruck in \cite{HofSp}; see also \cite{Croke} and \cite{ChGrTa}. The following lemma is proved in \cite[Lemma 2.20]{Va2}. Below 
$\Omega\subset M$ is a nonempty open set. 

\begin{lemma}\cite[Lemma 2.20]{Va2}
\label{moserite}
Suppose that $\norm{\theta}_{L^\infty} \leq 1$. Suppose that $s\in (0, r_S)$ is a constant and $x\in M$. 
Denote $B=B(x,s)$. Suppose that $u\in W^{1,p}_{loc} (M)$ is a function that is $\cA$-harmonic in the open 
set $\Omega \cap B$, satisfies $u-\theta \in W_0^{1,p}(\Omega)$, $\inf_{M}\theta\le u\le \sup_{M}\theta$, 
and $u=\theta$ a.e. in $M\setminus \Omega$. Then 
\[
\esssup_{B(x,s/2)} \varphi\left(|u-\theta|\right)^{p(n+1)}\leq c \int_{B} \varphi\left(|u-\theta|\right)^p, 
\]
where the constant $c$ is independent of $x$.
\end{lemma}
In Section~\ref{asdirmin} we will state and prove a similar estimate for solutions of the minimal graph 
equation.  

Next we show that the integral appearing in Lemma \ref{moserite} can be estimated from above by another 
integral that will be uniformly bounded provided sectional curvatures of $M$ are bounded as 
in Theorem~\ref{thmharm}. 
\begin{lemma}\label{mainlemma}
Let $M$ be a Cartan-Hadamard manifold of dimension $n\ge 2$. Suppose that 
\[
K(P)\leq-\frac{1+\varepsilon}{r(x)^{2}\log r(x)},
\]
for some constant $\varepsilon>0,$ where $K(P)$ is the sectional curvature of any plane $P\subset T_{x}M$
that contains the radial vector $\nabla r(x)$  
and $x$ is any point in $M\setminus B(o,R_0)$. Fix 
$\tilde{\varepsilon}\in (0,\varepsilon)$ and let $R_1\ge R_0$ be 
given by Proposition~\ref{jacest}. Suppose that $U\Subset M$ is an open, relatively compact set and that $u$ is an 
$\cA$-harmonic function in $U$, with $u-\theta\in W_0^{1,p}(U)$, where $\cA\in\cA^p(M)$, with 
\begin{equation}\label{pbound}
p<\frac{n\alpha}{\beta},
\end{equation}
and $\theta\in W^{1,\infty}(M)$ is a continuous function, with $\norm{\theta}_{\infty}\le 1$.
Then there exist a bounded $C^1$-function $\cC\colon [0,\infty)\to [0,\infty)$ and a constant $c_0\ge 1$ that is 
independent of $\theta,\ U,$ and $u$ such that 
\begin{align}\label{phiFest}
&\quad\ \int_U \varphi\bigl(|u-\theta|/c_0\bigr)^p  \bigl(\log(1+r)+\cC(r)\bigr) \\
\notag &\leq c_0+ c_0\int_U F\left(\frac{c_0 |\nabla \theta|r\log(1+r)}{\log(1+r)+\cC(r)}\right)
\bigl(\log(1+r)+\cC(r)\bigr).
\end{align}
\end{lemma}

\begin{proof}
We begin by proving a weighted Poincar\'e-type inequality. 
First of all, we have
\[
\Delta r\ge\frac{n-1}{r}
\]
in $M\setminus\{o\}$ since $M$ is a Cartan-Hadamard manifold. Moreover, by applying the standard Laplace 
comparison theorem and Proposition \ref{jacest}, we find that
\[
\Delta r(x)\geq (n-1)\left(\frac{1}{r(x)}+\frac{1+\tilde{\varepsilon}}{r(x)\log r(x)}\right)
\]
whenever $r(x)\ge R_1$. Therefore
\begin{equation}\label{estlaplacian}
r\log(1+r)\Delta r\geq (n-1)\left(\log(1+r)+\cE(r)\right)
\end{equation}
in $M$, where $\cE\colon [0,\infty)\to [0,\infty)$ is a bounded $C^1$-function satisfying 
\begin{equation}\label{E-defn}
\cE(r)=\begin{cases}
0, &\mbox{if }0\le r\le R_1;\\
\frac{(1+\tilde{\varepsilon})\log(1+r)}{\log r}, &\mbox{if }r\ge 2R_1.
\end{cases}
\end{equation}
By the assumption  \eqref{pbound}, we can choose $\delta>0$ such that
\begin{equation}\label{newpbound}
p<\frac{n\alpha}{(1+\delta)^2\beta}.
\end{equation}
Denote $h=|u-\theta|/c_0$, where the constant $c_0>0$ will be specified in due course.
Since $-1\le\inf_U\theta\le u\le\sup_U\theta\le 1$ in $U$, we may assume that $c_0$ is so 
large that $\norm{h}_{\infty}\le t_{\delta}$, where $t_{\delta}>0$ is a constant such that 
\eqref{(2.18)} and \eqref{(2.19)} hold for all $t\in (0,t_{\delta}]$.

Using \eqref{estlaplacian} and integration by parts, we obtain
\begin{align*}
&\quad\ (n-1) \int_U \varphi(h)^p \bigl(\log(1+r)+\cE(r)\bigr)\\
&\leq \int_U \varphi(h)^p r\log(1+r) \Delta r 
= -\int_U \left\langle \nabla \bigl(\varphi(h)^p r\log(1+r)\bigr),\nabla r  \right\rangle \\
&= -\int_U \varphi(h)^p\left(\tfrac{r}{1+r}+\log(1+r)\right) 
- p \int_U r \log(1+r)\varphi(h)^{p-1} \varphi^\prime (h) \langle \nabla h,\nabla r\rangle .
\end{align*}
This, together with H\"older's inequality, gives rise to
\begin{align*}
&\quad\ n\int_U \varphi(h)^p \left(\log(1+r)+\cC(r)\right)\\
&\leq  p\int_U r\log(1+r) \varphi(h)^{p-1} \varphi^\prime (h) |\nabla h|\\
&\leq  p \left( \int_U |\nabla h|^p \varphi^\prime (h)^p 
\frac{\bigl(r\log(1+r)\bigr)^p}{\bigl(\log(1+r)+\cC(r)\bigr)^{p-1}}\right)^{1/p}\\
&\quad\times \left(\int_U \varphi(h)^p  \bigl(\log(1+r)+\cC(r)\bigr)\right)^{(p-1)/p},
\end{align*}
where 
\begin{equation}\label{C-defn}
\cC(r)=\frac{r}{n(1+r)}+\frac{(n-1)\cE(r)}{n}.
\end{equation}
To simplify notation, we set
\begin{equation}\label{L-defn}
L(r)=\log(1+r)+\cC(r)
\end{equation}
and
\begin{equation}\label{w-defn}
w = \frac{r\log(1+r)}{\bigl(\log(1+r)+\cC(r)\bigr)^{(p-1)/p}}.
\end{equation}
Hence
\begin{equation}
\label{eqintlem2}
n \left(\int_U \varphi(h)^p L(r)\right)^{1/p}
\leq p \left( \int_U |\nabla h|^p \varphi^\prime (h)^p w^p \right)^{1/p}
\end{equation}
The gradient of $w$ is given by
\begin{equation}\label{eq-gradw}
\nabla w = L(r)^{1/p}\biggl(\frac{\log(1+r)+\tfrac{r}{1+r}}{L(r)}
+(\tfrac{1}{p}-1)\frac{r\log(1+r)\bigl(\tfrac{1}{1+r}+\cC^\prime (r)\bigr)}
{L(r)^2}\biggr)\nabla r.
\end{equation}
We claim that  
\begin{equation}\label{eq-gradw-est}
|\nabla w|\le L(r)^{1/p}
\end{equation}
for all $r$ large enough, say $r\ge R_2$, and 
\begin{equation}\label{eq-gradw-est2}
|\nabla w|\le c
\end{equation}
in $B(o,R_2)$.
To prove \eqref{eq-gradw-est}, we first note that $\cC'(r)r\to 0$ as $r\to\infty$, and therefore
\[
\log(1+r)+\frac{r}{1+r}-\frac{r\log(1+r)\bigl(\tfrac{1}{1+r}+\cC'(r)\bigr)}{L(r)}
\ge 0
\]
whenever $r$ is large enough. We have, for $r\ge R_2\ge R_1$,
\begin{align*}
0 &\le \frac{\log(1+r)+\tfrac{r}{1+r}}{L(r)} 
+(\tfrac{1}{p}-1)\frac{r\log(1+r)\bigl(\tfrac{1}{1+r}+\cC^\prime (r)\bigr)}
{L(r)^2} \\
&\le \frac{\log(1+r)+\tfrac{r}{1+r}}{L(r)} \\
&=  \frac{\log(1+r)+\tfrac{r}{1+r}}
{\log(1+r)+\tfrac{r}{1+r}+(\tfrac{1}{n}-1)\tfrac{r}{1+r}
+(1-\tfrac{1}{n})\tfrac{(1+\tilde{\varepsilon})\log(1+r)}{\log r}}\\
&\le 1 
\end{align*} 
since
\begin{align*}
&\quad\ \frac{\left(\tfrac{1}{n}-1\right)r}{1+r}
+\frac{\left(1-\tfrac{1}{n}\right)(1+\tilde{\varepsilon})\log(1+r)}{\log r}\\
&=\left(1-\frac{1}{n}\right)\left(\frac{(1+\tilde{\varepsilon})\log(1+r)}{\log r}-\frac{r}{1+r}\right)>0.
\end{align*}
Hence \eqref{eq-gradw-est} follows. The estimate \eqref{eq-gradw-est2} holds since $w$ is smooth in 
$M\setminus\{o\}$ and $w(r)/r\to 0$ as $r\to 0$.

Using the estimate $|\nabla h|\leq (|\nabla u|+|\nabla \theta |)/c_0 $, Minkowski's inequality,
and \eqref{(2.18)}, we obtain
\begin{align}\label{eqint2lem2}
&\quad\ \biggl(\int_U |\nabla h|^p \varphi^\prime (h)^p  w^p\biggr)^{1/p}\\
&\leq  c_0^{-1} \left(\int_U \left(\varphi^\prime (h) w |\nabla u|
+\varphi^\prime (h) w |\nabla \theta|\right)^p\right)^{1/p}\nonumber\\
\nonumber &\leq c_0^{-1} \left(\int_U \varphi^\prime (h)^p |\nabla u|^p w^p\right)^{1/p} 
+ c_0^{-1}\left(\int_U \varphi^\prime (h)^p |\nabla \theta|^p w^p\right)^{1/p}\\
&\leq  \frac{1+\delta}{c_0p^{1/p}} \left[\left(\int_U \psi^\prime (h) |\nabla u|^p w^p\right)^{1/p}
+\left(\int_U \psi^\prime (h) |\nabla \theta|^p w^p\right)^{1/p}\right]\nonumber.
\end{align}
Applying the Caccioppoli inequality~\eqref{ineq-caccio} with $u$ and $\theta$ replaced by $u/c_0$ 
and $\theta/c_0$, respectively, to the first term on the right-hand together with \eqref{(2.19)}, we obtain
\begin{align}
\label{eqint3lem2}
&\quad\ \biggl(\int_U w^p  \psi^\prime (h)  |\nabla u|^p\biggr)^{1/p}\\
\nonumber & \leq  \frac{\beta}{\alpha} \left(\int_U w^p \psi^\prime (h) |\nabla \theta|^p\right)^{1/p}
 +\frac{p\beta c_0}{\alpha} \left(\int_U \frac{\psi^p}{(\psi^\prime)^{p-1}}(h) |\nabla w|^p\right)^{1/p} \\
 &\le\frac{\beta}{\alpha} \left(\int_U w^p \psi^\prime (h) |\nabla \theta|^p\right)^{1/p}
	+ \frac{p^{1/p}\beta c_0(1+\delta)}{\alpha} \left(\int_U \varphi(h)^p |\nabla w|^p\right)^{1/p}.\nonumber\hspace{-1ex}
\end{align}
Now combining \eqref{eqintlem2}, \eqref{eqint2lem2}, and \eqref{eqint3lem2}, we find
\begin{align*}
&\quad\ n \left(\int_U \varphi(h)^p L(r)\right)^{1/p}
\leq  p \left(\int_U |\nabla h|^p \varphi^\prime (h)^p w^p\right)^{1/p}\\
&\leq  (1+\delta ) p^{1-\tfrac 1p}c_0^{-1} 
\left[\left(\int_U \psi^\prime (h) |\nabla u|^p w^p\right)^{1/p}  
+ \left(\int_U \psi^\prime (h) |\nabla \theta|^p w^p\right)^{1/p}\right]\\
&\leq (1+\delta )p^{1-1/p}c_0^{-1}\left[(1+\tfrac{\beta}{\alpha})
\left(\int_U \psi^\prime (h)|\nabla \theta|^p w^p\right)^{1/p} \right.\\
&\quad\left. +\frac{p^{1/p}\beta c_0 (1+\delta)}{\alpha}\left(\int_U \varphi(h)^p |\nabla w|^p
\right)^{1/p}\right]\\
&\leq (1+\delta ) p^{1-1/p}c_0^{-1}(1+\tfrac{\beta}{\alpha})\left(\int_U \psi^\prime (h)|\nabla 
\theta|^p w^p\right)^{1/p} \\
&\quad +\frac{p\beta (1+\delta)^2}{\alpha}\left(\int_U \varphi(h)^p L(r)\right)^{1/p}+ C,
\end{align*}
where in the last step we used \eqref{eq-gradw-est} and \eqref{eq-gradw-est2} to estimate
\begin{align*}
\int_U \varphi(h)^p |\nabla w|^p 
& = \int_{U\cap B(o,R_2)} \varphi(h)^p |\nabla w|^p 
+ \int_{U\setminus B(o,R_2)} \varphi(h)^p |\nabla w|^p\\
&\le \tilde{C}+\int_U \varphi(h)^p L(r).
\end{align*}
Since
\[
p<\dfrac{n\alpha}{(1+\delta)^2\beta},
\]
it follows that there exists a constant $C$ depending on $p,n, \alpha,\beta$ such that
\begin{equation}\label{eq-phiphiprime}
\int_U \varphi(h)^p L(r) 
\leq  C \int_U \varphi^\prime (h)^p |\nabla \theta|^p w^p +C_0.
\end{equation}
Next, recalling that $F(\cdot^{1/p})$ and $G(\cdot^{1/p})^p$ are complementary Young functions, we have, for all $x,y\geq 0$ and $k>0$,
\begin{align}\label{youngineq}
	x y &=kx(y/k)\leq k\left(G(x^{1/p})^p + F(k^{-1/p}y^{1/p})\right)\\
	\notag &= kG(x^{1/p})^p + kF(k^{-1/p}y^{1/p}).
\end{align}
The definition of $w$, previous inequalities \eqref{eq-phiphiprime}, \eqref{youngineq}, and 
\eqref{relvarphi1} yield 
\begin{align*}
\int_U \varphi(h)^p  L(r) &\leq
C \int_U \varphi^\prime (h)^p L(r)
\left(\frac{|\nabla \theta|r\log(1+r)}{L(r)}\right)^p +C_0\\
&\leq  Ck \int_U G\bigl(\varphi^\prime (h)\bigr)^p L(r)\\
&\quad + Ck\int_U F \left(\frac{k^{-1/p}|\nabla \theta|r\log(1+r)}{L(r)}\right) L(r) + C_0\\
&=  Ck \int_U \varphi(h)^p L(r)\\
&\quad + Ck\int_U F \left(\frac{k^{-1/p}|\nabla \theta|r\log(1+r)}{L(r)}\right) 
L(r) + C_0.
\end{align*}
Taking $k>0$ small enough, we finally obtain
\[
\int_U \varphi(h)^p L(r) 
\leq \frac{Ck}{1-Ck}\int_U F \left(\frac{k^{-1/p}|\nabla \theta|r\log(1+r)}
{L(r)}\right)L(r)+\frac{C_0}{1-Ck}.
\]
\end{proof}

We are now in position to prove Theorem~\ref{thmharm}. In fact, we prove the following localized version concerning 
the $\cA$-regularity of a point $x_0\in\partial_\infty M$ which then implies Theorem~\ref{thmharm} since the uniqueness statement follows from the comparison principle.
\begin{theorem}\label{thm-A-regu}
Let $M$ be a Cartan-Hadamard manifold of dimension $n\ge 2$. Suppose that 
\begin{equation}\label{curv_ass_minim2b}
-\frac{\bigl(\log r(x)\bigr)^{2\bar{\varepsilon}}}{r(x)^2}\leq K(P)\leq
-\frac{1+\varepsilon}{r(x)^{2}\log r(x)},
\end{equation}
for some constants $\varepsilon>\bar{\varepsilon}>0,$ where $K(P)$ is the sectional
curvature of any plane $P\subset T_{x}M$ that contains the radial vector $\nabla r(x)$ and $x$ is any point 
in a cone neighborhood $U$ of $x_0\in \partial_\infty M$.
Then $x_0$ is $\cA$-regular for every $\cA\in\cA_p(M)$, with $1< p< n\alpha/\beta$.
\end{theorem}

\begin{proof}
Let $f\colon \partial_\infty M\to \R$ be a continuous function. To prove that $x_0$ is $\cA$-regular, we 
need to show that
\[
\lim_{x\to x_0} \overline{H}_f (x)=f(x_0).
\]
Fix an arbitrary $\varepsilon' >0$. Let $v_0=\dot{\gamma}_0^{o,x_0}$ be the initial vector of the geodesic 
ray from $o$ to $x_0$. Furthermore, let $\delta\in (0,\pi)$ and $R_0>0$ be such that 
$T(v_0,\delta ,R_0)\subset U$ and that $|f(x_1)-f(x_0)|<\varepsilon'$ for all 
$x_1\in C(v_0,\delta)\cap \partial_\infty M$; see \ref{subsec_hada} for the notation. 
Next we fix $\tilde{\varepsilon}\in (\bar{\varepsilon},\varepsilon)$, where 
$\varepsilon>\bar{\varepsilon}>0$ are the constants in the curvature assumption 
\eqref{curv_ass_minim2b}. Let $r_1>\max(2,R_1)$, where $R_1\ge R_0$ is given by 
Proposition \ref{jacest}. We denote $\Omega=T(v_0,\delta,r_1)\cap M$ and define $\theta \in C(\bar{M})$ by 
setting
\[
\theta (x)=\min \left(1,\max \left(r_1+1-r(x),\delta^{-1}\sphericalangle_o (x_0,x)\right)\right).
\]
Note that $\theta=1$ on $\partial\Omega$.
Let $\Omega_j=\Omega\cap B(o,j)$ for integers $j>r_1$ and let $u_j$ be the unique $\cA$-harmonic function in 
$\Omega_j$ with $u_j-\theta\in W^{1,p}_{0}(\Omega_j)$. It is clear that each $y\in\partial\Omega_j$ is $\cA$-regular
and hence $u_j$ can be continuously extended to $\partial\Omega_j$ by setting $u_j=\theta$ on $\partial\Omega_j$. 
Since $0\le u_j\le 1$, the sequence $(u_j)$ is equicontinuous, and therefore by the Ascoli-Arzel\'a theorem, there 
exists a subsequence, still denoted by $(u_j)$, that converges locally uniformly to a continuous function 
$u\colon\bar\Omega\to [0,1]$. It follows that $u$ is $\cA$-harmonic in $\Omega$; see e.g. 
\cite[Chapter 6]{HKM} for these  
boundary regularity and convergence results. 
Next we prove that
\begin{equation}\label{(3.18II)}
\lim_{ \stackrel{ \mbox{\scriptsize$x\to x_0$} }{x\in\Omega}}u(x)=0.
\end{equation}
Denote $\tilde{\theta}=\theta/c_0,\ \tilde{u}_j=u_j/c_0$, and $\tilde{u}=u/c_0$, where $c_0$ is given by 
Lemma~\ref{mainlemma}. 
Fatou's lemma and Lemma~\ref{mainlemma} applied to 
$U=\Omega_{j}$ imply that
\begin{align}
\label{thme1}
\int_\Omega \varphi (|\tilde{u}-\tilde{\theta}|)^p 
&= \int_\Omega \varphi (|u -\theta|/c_0)^p 
\leq \liminf_{j\to \infty} \int_{\Omega_j} 
\varphi (|u_j -\theta|/c_0)^p\\
&\leq \liminf_{j\to \infty} \int_{\Omega_j} 
\varphi (|u_j -\theta|/c_0)^p L(r)\nonumber\\
&\leq c_0 + c_0 \int_\Omega 
F\left(\frac{c_0 |\nabla \theta|r\log(1+r)}{L(r)}\right)L(r).
\nonumber
\end{align} 
We will show at the end of the proof that the right-side in \eqref{thme1} is finite.
Meanwhile we extend each $u_j$ to a function $u_j\in W_{\loc}^{1,p}(M)\cap C(M)$ by setting 
$u_j(y)=\theta(y)$ for every $y\in M\setminus \Omega_j$. Let $x\in\Omega$ and fix 
$s\in (0,r_S)$. If $j$ is large enough, $\tilde{u}_j$ satisfies the 
assumption of Lemma~\ref{moserite} and hence 
\[
\sup_{B(x,s/2)} \varphi (|\tilde{u}_j - \tilde{\theta}|)^{p(n+1)}
\leq c\int_{B(x,s)} \varphi (|\tilde{u}_j - \tilde{\theta}|)^p.
\]
Note that we may replace $\esssup$ by $\sup$ because $u_j-\theta$ is continuous in $M$. The dominated 
convergence theorem implies that
\begin{align}\label{(3.19II)}
\sup_{B(x,s/2)} \varphi (|\tilde{u}-\tilde{\theta}|)^{p(n+1)}
&= \sup_{B(x,s/2)}\lim_{j\to \infty} \varphi (|\tilde{u}_j-\tilde{\theta}|)^{p(n+1)}\\
\nonumber &\leq \limsup_{j\to \infty}\sup_{B(x,s/2)}\varphi (|\tilde{u}_j-\tilde{\theta}|)^{p(n+1)}\\
&\leq c \limsup_{j\to \infty} \int_{B(x,s)} \varphi(|\tilde{u}_j-\tilde{\theta}| )^p\nonumber\\
&= c\int_{B(x,s)} \varphi (|\tilde{u}-\tilde{\theta}|)^p.\nonumber
\end{align}
Let $(x_k)$ be a sequence of points in $\Omega$ so that $x_k \to x_0$ as $k\to\infty$. Applying the 
estimate \eqref{(3.19II)} above with $x=x_k$ and a fixed $s\in (0,r_S)$ and assuming that the right-side of \eqref{thme1} is finite we obtain
\[
\lim_{k\to \infty}\sup_{B(x_k,s/2)} \varphi (|\tilde{u}-\tilde{\theta}|)^{p(n+1)}
\le c \lim_{k\to \infty}
\int_{B(x_k,s)} \varphi (|\tilde{u}-\tilde{\theta}|)^p=0.
\]
Hence
\[
\lim_{k\to\infty}|\tilde{u}(x_k)-\tilde{\theta}(x_k)|=0
\]
and, consequently, \eqref{(3.18II)} holds.
Next we define $w\colon M\to \R$ by
\[
w(x)=\begin{cases}
\min \bigl(1, 2u(x)\bigr),&\mbox{if } x\in \Omega; \\
1,&\mbox{if } x\in M\setminus \Omega.
\end{cases}
\]
Then $w$ is $\cA$-superharmonic in $M$ (see \cite[Lemma 7.2]{HKM}) and hence, by the definition of 
$\overline{H}_f$, we have
\[
\overline{H}_f \leq f(x_0) +\varepsilon' +2(\sup |f|)w.
\]
Hence, by \eqref{(3.18II)}
\[
\limsup_{x\to x_0} \overline{H}_f (x) \leq f(x_0) +\varepsilon'.
\] 
One can prove in a similar way that 
\[
\liminf_{x\to x_0} \underline{H}_f (x)\geq f(x_0)-\varepsilon'.
\] 
We deduce that 
\[
\lim_{x\to x_0} \overline{H}_f (x)=f(x_0),
\] 
and therefore $x_0$ is $\cA$-regular.

To conclude the proof, it remains to show that
\begin{equation}
\label{thmlasint}
\int_\Omega 
F\left(\frac{c_0 |\nabla \theta|r\log(1+r)}{L(r)}\right)L(r)<\infty.
\end{equation}
Recall that above $\Omega=T(v_0,\delta,r_1)\cap M$, with $v_0=\dot{\gamma}_0^{o,x_0}$.
The integral~\eqref{thmlasint} will be estimated from above by using geodesic polar coordinates
$(r,v)$ for points $x\in\Omega$. Here $r=r(x)\in [r_1,\infty)$ and $v=\dot{\gamma}_0^{o,x}$. 
Let $\lambda(r,v)$ be the Jacobian for these polar coordinates. We need to estimate $\lambda$ and the 
function $F$ from above. To this end, let $a,b\colon [0,\infty)\to [0,\infty)$ be smooth functions such 
that they are constant in some neighborhood of $0$, 
\[
-b^2\bigl(r(x)\bigr)\le K(P) \le -a^2\bigl(r(x)\bigr)
\]
for all $x\in C(v_0,\delta)$ and for all $2$-planes $P\subset T_x M$ containing the radial vector 
$\nabla r$, and that
\begin{align*}
a^2(t)&=\frac{1+\varepsilon}{t^{2}\log t},\\
b^2(t)&=\frac{\left(\log t\right)^{2\bar{\varepsilon}}}{t^2}
\end{align*}
for $t\ge R_0$. For $x\in \Omega$, we denote by $J(x)$ the supremum and by $j(x)$ the 
infimum of $|V(r(x))|$ over all Jacobi fields $V$ along $\gamma^{o,x}$ that satisfy $V_0=0$, 
$|V_0^\prime|=1$, and $V_0^\prime \bot \dot{\gamma}_0^{o,x}$. By applying the Rauch comparison 
theorem we get the estimates 
\begin{align}\label{eq-Jac-estim}
j(x)&\geq f_a (r(x));\\
J(x)&\leq f_b (r(x)),
\end{align}
where $f_a$ and $f_b$ are the solutions to corresponding Jacobi equations~\eqref{Jacobi_eq}; see e.g. 
\cite[Proposition 2.5]{HoVa}. Thus we have
\begin{equation}\label{eq-jacdet-est}
\lambda(r,v)\le f_b(r)^{n-1}
\end{equation}
for all points $x=(r,v)\in\Omega$. We also recall from \cite[Lemma 2]{Va1} that 
\begin{equation}\label{grad-theta-est}
|\nabla \theta (x)|\leq \frac{c}{j(x)}\leq \frac{c}{f_a (r(x))}
\end{equation}
in $\Omega$. It follows that there exists a constant $c_1$ such that
\begin{equation}\label{eq-integrand-est}
\frac{c_0 |\nabla \theta|r\log(1+r)}{L(r)}=
\frac{c_0 |\nabla \theta|r\log(1+r)}{\log(1+r)+\cC(r)}\le\frac{r}{c_1 f_a(r)}
\end{equation}
for all $r$ large enough.
Since the functions $\varphi$ and $F\in\cF_p$ were fixed so that $F$ satisfies \eqref{eq-F-upper}, 
we have in particular, $F\leq \tilde{F}$, where
\[
\tilde{F}(s)=\exp \left(-\frac{1}{s}\left(\log \frac{1}{s}\right)^{-1-\varepsilon_0}\right)
\]
for all $s$ small enough and $\varepsilon_0\in (0,1)$. 
In what follows, we assume that $t_0\ge R_1$ is a sufficiently large constant.
For $t\ge t_0$, we define
\begin{align*}
\Phi (t)&= \left(t^2 \tilde{F} \left(\frac{t}{c_1 f_a(t)}  \right)\right)^{\frac{1}{1-n}}\\
&= t^{-\frac{2}{n-1}} \exp \left(\frac{1}{n-1} \frac{c_1 f_a (t)}{t} \left(\log \frac{c_1 f_a (t)}{t} 
\right)^{-1-\varepsilon_0} \right),
\end{align*}
and thus
\begin{align*}
\frac{\Phi^\prime (t)}{\Phi (t)}
&= \frac{-2t + c_1 \left(1-(1+\varepsilon_0) \left(\log \frac{c_1 f_a (t)}{t}\right)^{-1} \right)
}{(n-1) t^2}\\
&\quad \times\bigl(t f_a^\prime (t) -f_a (t)\bigr)\left(\log \frac{c_1 f_a (t)}{t}\right)^{-1-\varepsilon_0} 
\end{align*}
Straightforward computations, using Proposition~\ref{jacest}, yield to
\[
\left(\frac{t f_a^\prime (t)}{f_a(t)}-1\right)\frac{f_a(t)}{t}
\ge (1+\tilde{\varepsilon})\bigl(\log t\bigr)^{\tilde{\varepsilon}}
\]
for all $t\ge R_1$. It follows that
\[
\frac{\Phi^\prime (t)}{\Phi (t)} \ge \frac{2\bigl(\log t\bigr)^{\bar{\varepsilon}}}{t}= 2b(t)
\]
for all $t\ge t_0$.
Since $b^\prime (t)/b(t)^2\to 0$ as $t\to \infty$, we obtain
\[
\lim_{t\to \infty} \frac{f_b^\prime (t)}{b(t) f_b (t)}=1
\]
by \cite[Lemma 2.3]{HoVa}.
Therefore we have
\[
\frac{\Phi^\prime (t)}{\Phi (t)}\geq 2 b(t) \geq \frac{f_b^\prime (t)}{f_b(t)}
\] 
for $t\ge t_0$. It follows that $\Phi (t)\geq cf_b(t)$, for all $t\ge t_0$.
Thus we have
\begin{align*}
&\quad\ F\left(\frac{c_0 |\nabla \theta(r,v)|r\log(1+r)}{L(r)}\right)L(r)\lambda(r,v)\\
&=F\left(\frac{c_0 |\nabla \theta(r,v)|r\log(1+r)}{\log(1+r)+\cC(r)}\right)\bigl(\log(1+r)+\cC(r)\bigr)
\lambda(r,v)\\
&\leq c
\tilde{F}\left(\frac{r}{c_1 f_a(r)}\right)
\bigl(\log(1+r)+\cC(r)\bigr)\Phi(r)^{n-1}\\
& = c\bigl(\log(1+r)+\cC(r)\bigr)r^{-2}
\end{align*} 
for all $x=(r,v)\in U\cap M$ outside a compact set. Since $\cC$ is a bounded function, this shows that \eqref{thmlasint} holds and therefore 
concludes the proof of Theorem \ref{thm-A-regu}.
\end{proof}

\section{Dirichlet problem at infinity for the minimal graph\\ equation}\label{asdirmin}
In this section we will prove Theorem~\ref{thmmin}. We will use a slightly different approach than 
the one adopted in the proof of Theorem~\ref{thmharm} but the main ingredients will be the same. 
However, to solve the Dirichlet problem at infinity for the minimal graph equation, some extra 
difficulties appear. The first one is the fact that the minimal graph operator does not satisfy 
\eqref{aharmstruc}. Therefore, we need to adapt the previous Caccioppoli inequality proved in 
Lemma~\ref{caccio}. The second difficulty is linked to the fact that it may not be possible, in general, 
to solve the minimal graph equation on the sets $\Omega_j$ as defined in the proof of 
Theorem~\ref{thmharm}. 

\subsection{Caccioppoli inequality and some consequences}
We begin this section with the following Caccioppoli-type inequality. In what follows we use the customary 
notation $W(x)=\sqrt{1+|\nabla u(x)|^2}$ for a smooth solution $u$ of the minimal graph equation.
\begin{lemma}\label{newlem-mincaccio}
Suppose that $\Psi\colon [0,\infty)\to [0,\infty)$ is a homeomorphism that is smooth on $(0,\infty)$.
Let $U\Subset M$ be an open and relatively compact set. Suppose that $\eta \geq 0$ is a locally 
Lipschitz function on $U\setminus\{o\}$. Suppose that $\theta,u\in L^\infty (U)\cap W^{1,2}(U)$ are 
continuous functions and that $u\in C^2(U)$ is a solution of the minimal graph equation in $U$. Denote 
\[
h=\frac{|u-\theta|}{\nu},
\] 
where $\nu>0$ is a constant, and suppose that
\[
\eta^2 \Psi (h)W\in W^{1,2}_0 (U).
\]
Then we have
\begin{align}
\label{cacciomin}
\int_U \eta^2 \Psi^\prime (h)|\nabla u|^2
&\leq 4 \int_U \eta^2 \Psi^\prime (h)|\nabla \theta|^2
+8\nu^2 \int_U \frac{\Psi^2}{\Psi^\prime}(h) |\nabla \eta|^2\\
&\quad +4\nu^2\int_U \eta^2\frac {\Psi^2}{\Psi^\prime}(h)|\nabla\log W|^2.\nonumber
\end{align}
\end{lemma}
\begin{proof}
We begin by defining 
\[
f=\nu\eta^2 \Psi \left((u-\theta)^+ /\nu \right)W
-\nu \eta^2 \Psi \left((u-\theta)^- /\nu\right)W.
\]
It is easy to see that $f\in W^{1,2}_0(U)$ and its gradient is given by
\begin{align*}
\nabla f &=\eta^2 \Psi^\prime (h) W (\nabla u-\nabla\theta) + 
2\nu\eta\sgn(u-\theta) \Psi (h) W\nabla \eta\\
&\quad +\nu\eta^2\sgn(u-\theta)\Psi(h)\nabla W.
\end{align*}
Using $f$ as a test function in the minimal graph equation, we obtain 
that
\begin{align*}
\int_U  \eta^2 \Psi^\prime (h)|\nabla u|^2
&= \int_U \eta^2 \Psi^\prime (h)\langle\nabla u,\nabla\theta\rangle
-2\nu\int_U \sgn(u-\theta)\eta\Psi(h)\langle\nabla u,\nabla\eta\rangle\\
&\quad-\nu\int_U \sgn(u-\theta)\eta^2\Psi(h)\langle\nabla \log W,\nabla u\rangle\\
&\leq  \int_U \eta^2\Psi^\prime (h)|\nabla u||\nabla\theta| 
+2\nu\int_U \eta\Psi(h)|\nabla u||\nabla\eta| \\
&\quad+\nu\int_U\eta^2 \Psi(h)|\nabla u||\nabla\log W|.
\end{align*}
We estimate the terms on the right-side by Young's inequality with $\varepsilon$ as 
\begin{align*}
&\int_U \eta^2\Psi^\prime (h)|\nabla u||\nabla\theta|  \le 
\varepsilon/2 \int_U \eta^2 \Psi^\prime (h)|\nabla u|^2
+ 1/(2\varepsilon)\int_U \eta^2 \Psi^\prime (h)|\nabla \theta|^2,\\
& 2\nu\int_U \eta\Psi(h)|\nabla u||\nabla\eta|
\le\varepsilon \int_U \eta^2 \Psi^\prime (h)|\nabla u|^2 +
\nu^2 /\varepsilon  \int_U \frac{\Psi^2}{\Psi^\prime}(h) |\nabla \eta|^2,
\end{align*}
and
\begin{align*}
\nu\int_U\eta^2 \Psi(h)|\nabla u||\nabla\log W| 
&\le \varepsilon/2\int_U \eta^2\Psi^{\prime}(h)|\nabla u|^2\\
&\quad +
\nu^2/(2\varepsilon)\int_U \eta^2\frac {\Psi^2}{\Psi^\prime}(h)|\nabla\log W|^2.
\end{align*}
Choosing $\varepsilon=1/4$ above proves the claim.
\end{proof}
\begin{remark}
As can be seen later in the proof of Lemma~\ref{2ndmainlemma}, the second term 
\[
8\nu^2 \int_U \frac{\Psi^2}{\Psi^\prime}(h) |\nabla \eta|^2
\]
on the right-side of \eqref{cacciomin} is the only term that affects the dimension restriction
$n\ge 3$ in Theorem~\ref{thmmin}. One could improve the factor $8\nu^2$ to $(4+\epsilon)\nu^2$ for any 
$\epsilon>0$ but, nevertheless, the dimension bound $n\ge 3$ still remains.
\end{remark}
Before we state and prove a counterpart of Lemma~\ref{moserite} for the minimal graph equation,
we recall from \ref{subsec_young} that $\varphi\colon [0,\infty)\to [0,\infty)$ is a homeomorphism, smooth 
on $(0,\infty)$, and satisfies \eqref{relvarphi1}, i.e.
\[
G \circ \varphi^\prime =\varphi,
\]
where the homeomorphic Young function $G\colon [0,\infty)\to [0,\infty)$ is, in particular, convex. Hence there exist 
positive constants $t_1$ and $c_2$ such that
\begin{equation}\label{good-varphi-prop}
\varphi(t)\le 1,\quad \varphi^\prime(t)\le 1,\quad \text{and}\quad \varphi(t)\le c_2\varphi^\prime (t) 
\end{equation}
for all $t\in (0,t_1]$.
\begin{lemma}
\label{moseritemin}
Let $\Omega=B(o,R)$ and suppose that $\theta\in C^1(\Omega)$ with $\norm{\theta}_{C^1(\Omega)}\leq C_1$. 
Let $u\in C^2(\Omega)$ be a solution of the minimal graph equation in $\Omega$ such that
$\inf_{\Omega}\theta\le u\le \sup_{\Omega}\theta$ and $|\nabla\log W|\leq C_1$. Fix $s\in (0, r_S)$, where $r_S$ is the radius in the Sobolev inequality~\eqref{ineq-sobo},  
and suppose that $B=B(x,s)\subset\Omega$. Then there exists a positive constant 
$\nu_0=\nu_0(\varphi,C_1)$ such that for all $\nu\ge\nu_0$
\[
\sup_{B(x,s/2)} \varphi\left(|u-\theta|/\nu\right)^{2(n+1)}
\leq c_3 \int_{B} \varphi\left(|u-\theta|/\nu\right)^2, 
\]
where $c_3$ is a positive constant depending only on 
$n, \nu, s, C_S, C_1$ and $\varphi$.
\end{lemma}

\begin{proof}
We denote $\kappa=n/(n-1),\ B/2=B(x,s/2)$, and $h=|u-\theta|/\nu$, where $\nu\ge\nu_0>0$ will be fixed 
in due course. 
For each $j\in\N$ we denote $s_j=s(1+\kappa^{-j})/2$ and $B_j=B(x,s_j)$. Furthermore, let $\eta_j$ be a 
Lipschitz function such that
$0\le\eta_j \le 1,\ \eta_j|B_{j+1}\equiv 1,\ \eta_j|(M\setminus B_j)\equiv 0$, and that
\[
|\nabla\eta_j|\le \frac{1}{s_j-s_{j+1}}=2n\kappa^j/s.
\]
For $\Phi=\varphi^2$ and $m\ge 1$ we have 
\[
\bigl|\nabla\bigl(\eta_j^2\Phi(h)^m\bigr)\bigr| \le 
2\eta_j\Phi(h)^m|\nabla\eta_j| + m\eta_j^2\Phi^\prime(h)\Phi(h)^{m-1}|\nabla h|.
\]
We claim that
\begin{equation}\label{moser-claim1}
\left(\int_{B_{j+1}}\Phi(h)^{\kappa m}\right)^{1/\kappa}
\le c(\kappa^j + m + \kappa^{2j} /m ) \int_{B_j}\Phi(h)^{m-1}
\end{equation}
for all $\nu\ge\nu_0$, with $\nu_0=\nu_0(\varphi,C_1)$ large enough.
For every $m,j\ge 1$, $\eta_j^2\Phi(h)^m$ is a Lipschitz function supported in $\bar{B_j}$. 
By the Sobolev inequality~\eqref{ineq-sobo} 
we first have
\begin{align}\nonumber
\biggl(\int_{B_{j+1}} \Phi(h)^{\kappa m}\biggr)^{1/\kappa} &\le
\left(\int_{B_j}\left(\eta_j^2 \Phi(h)^m\right)^\kappa\right)^{1/\kappa}
\le C_S\int_{B_j}\bigl|\nabla\bigl(\eta_j^2 \Phi(h)^m\bigr)\bigr|\\
&\le 2C_S\int_{B_j}\eta_j\Phi(h)^m|\nabla\eta_j| + 
C_S\int_{B_j}\eta_j^2\bigl(\Phi^m\bigr)^\prime (h)|\nabla h|\nonumber \\
&\le c\kappa^j\int_{B_j}\Phi(h)^m + 
C_S/\nu\int_{B_j}\bigl(\Phi^m\bigr)^\prime (h)|\nabla\theta | \label{moser1step}\\
&\quad +
C_S/\nu\int_{B_j}\eta_j^2\bigl(\Phi^m\bigr)^\prime (h)|\nabla u|\nonumber. 
\end{align}
Next we use the assumption
\[
-C_1\le\inf_{\Omega}\theta\le u\le\sup_{\Omega}\theta\le C_1
\]
to observe that $|u-\theta|\le 2C_1$. Hence, by \eqref{good-varphi-prop}, we can choose 
$\nu_0$ large enough so that 
\[
\varphi(h)\le 1,\quad \varphi^\prime(h)\le 1,\quad \text{and}\quad \varphi(h)\le c_2\varphi^\prime (h) 
\]
for $\nu\ge\nu_0$. Consequently,
\begin{equation}\label{phiphi}
\Phi(h)\le 1,\quad \Phi^\prime (h)\le 2,\quad \text{and}\quad \Phi(h)\le \tfrac{c_2}{2}\Phi^\prime (h).
\end{equation}
We obtain estimates
\begin{equation}\label{1st-term}
\int_{B_j}\Phi(h)^m \le \int_{B_j}\Phi(h)^{m-1}
\end{equation}
and
\begin{equation}\label{2nd-term}
\int_{B_j}\bigl(\Phi^m\bigr)^\prime (h)|\nabla\theta |
=m\int_{B_j}\Phi(h)^{m-1}\Phi^\prime(h)|\nabla\theta|\le 2m C_1\int_{B_j}\Phi(h)^{m-1}.
\end{equation}
We estimate the third term on the right-side of \eqref{moser1step} first as
\begin{align}\label{3rd-term_a}
\int_{B_j}\eta_j^2\bigl(\Phi^m\bigr)^\prime (h)|\nabla u|
&\le \int_{B_j}\eta_j^2\bigl(\Phi^m\bigr)^\prime (h)(1+|\nabla u|^2)\\
&\le 2m\int_{B_j}\Phi (h)^{m-1} + 
\int_{B_j}\eta_j^2\bigl(\Phi^m\bigr)^\prime (h)|\nabla u|^2.\nonumber
\end{align}
Next we notice that $\eta_j^2\Phi (h)^m W\in W^{1,2}_0(B_j)$ since 
$\supp\eta_j\subset\bar{B_j}$. Thus we may apply 
the Caccioppoli-type inequality~\eqref{cacciomin} with $\Psi=\Phi^m$
to obtain
\begin{align}
\int_{B_j}\eta_j^2\bigl(\Phi^m\bigr)^\prime (h)|\nabla u|^2
&\leq 4 \int_{B_j} \eta_j^2(\Phi^m)^\prime (h)|\nabla \theta|^2
+8\nu^2 \int_{B_j} \frac{\Phi^{2m}}{(\Phi^m)^\prime}(h) |\nabla \eta_j|^2\nonumber \\
&+4\nu^2\int_{B_j} \eta_j^2\frac {\Phi^{2m}}{(\Phi^m)^\prime}(h)|\nabla\log W|^2 \nonumber\\
&\leq c(m+\kappa^{2j}/m +1/m)\int_{B_j}\Phi(h)^{m-1}.\label{3rd-term2}
\end{align}
Now the estimate~\eqref{moser-claim1} follows by inserting estimates \eqref{1st-term}--\eqref{3rd-term2} into \eqref{moser1step}. We apply \eqref{moser-claim1} with $m=m_j+1$, where $m_j=(n+1)\kappa^j-n$. 
Since $m_{j+1}=\kappa (m_j+1)$, \eqref{moser-claim1} takes the form
\[
\left(\int_{B_{j+1}}\Phi (h)^{m_{j+1}}\right)^{1/\kappa}\leq C\kappa^j \int_{B_j} \Phi (h)^{m_j}.
\]
By denoting 
\[
I_j=\left(\int_{B_j} \Phi (h)^{m_j} \right)^{1/\kappa^j},
\] 
we can write the previous inequality as 
\[
I_{j+1}\leq C^{1/\kappa^j}\kappa^{j/\kappa^j}I_j.
\]
Since
\[
\limsup_{j\to \infty} I_j\geq \lim_{j\to \infty} 
\left(\int_{B/2} \Phi(h)^{m_j} \right)^{(n+1)/m_j}
=\sup_{B/2}\Phi (h)^{n+1},
\]
we finally get
\[
\sup_{B/2}\Phi (h)^{n+1}\leq \limsup_{j\to \infty} I_j\leq 
C^{n} \kappa^{S}I_0
\leq c\int_B  \Phi (h), 
\]
where
\[
S=\sum_{j=0}^\infty j\kappa^{-j}<\infty.
\]
\end{proof}
Next we will prove the counterpart of Lemma \ref{mainlemma}. We point out that some extra difficulties 
will appear due to the presence of $|\nabla\log W|$ in the right-side of the Caccioppoli 
inequality \eqref{cacciomin}. 
Moreover, we have to assume that the dimension of $M$ is at least $3$.
Let us recall the definitions of the bounded $C^1$-function 
$\cC\colon [0,\infty)\to [0,\infty)$ from \eqref{C-defn} and \eqref{E-defn} and functions
\[
L(r)=\log(1+r)+\cC(r)
\]
and
\[
w = \frac{r\log(1+r)}{\sqrt{L(r)}}.
\]
from \eqref{L-defn} and \eqref{w-defn} (with $p=2$), respectively.
\begin{lemma}\label{2ndmainlemma}
Let $M$ be a Cartan-Hadamard manifold of dimension $n\ge 3$. Suppose that 
\[
K(P)\leq-\frac{1+\varepsilon}{r(x)^{2}\log r(x)},
\]
for some constant $\varepsilon>0,$ where $K(P)$ is the sectional curvature of any plane $P\subset T_{x}M$
that contains the radial vector $\nabla r(x)$  
and $x$ is any point in $M\setminus B(o,R_0)$. Fix 
$\tilde{\varepsilon}\in (0,\varepsilon)$ and let $R_1\ge R_0$ be 
given by Proposition~\ref{jacest}. Let $U=B(o,R)$, with $R>R_1$, and suppose that $u\in C^{2}(\bar{U})$ 
is the unique solution of the minimal graph equation \eqref{Mequ} in $U$, with 
$u|\partial U=\theta|\partial U$, where $\theta\in C^\infty(M)$, with $\norm{\theta}_{\infty}\!\le\! C$. 
Furthermore, suppose that $|\nabla\log W(x)|\!\le\!\cW\bigl(r(x)\bigr)$, where 
$\cW\colon [0,\infty)\to [0,\infty)$ is a continuous function that is independent of $u$, and 
$\cW(r)=o(1/r)$ as $r\to\infty$.
Then there exists a constant $c_4\ge 1$ that is independent of $u$ such that 
\begin{equation}
\label{2ndphiFest}
\int_U \varphi\bigl(|u-\theta|/c_4\bigr)^2  L(r) 
\leq c_4+ c_4\int_U F\left(\frac{c_4 |\nabla \theta|r\log(1+r)}{L(r)}\right)L(r).
\end{equation}
\end{lemma}
\begin{proof}
As in the proof of Lemma~\ref{mainlemma} we denote $h=|u-\theta|/\nu$, where $\nu\ge\nu_0$ will be fixed 
later. Recall from \eqref{eqintlem2} with $p=2$ that
\[
n \left(\int_U \varphi(h)^2 L(r) \right)^{1/2}
\leq 2 \left( \int_U |\nabla h|^2 \varphi^\prime (h)^2 w^2
\right)^{1/2}.
\]
We estimate the right-side as
\begin{align}\label{(4.10)}
&\quad\ 2 \left( \int_U |\nabla h|^2 \varphi^\prime (h)^2 w^2 \right)^{1/2}\\
&\le  2/\nu\left(\int_U 
\bigl(\varphi^\prime (h)|\nabla u| w +\varphi^\prime (h)|\nabla\theta|w\bigr)^2
\right)^{1/2}\nonumber\\
&\le 2/\nu\left(\int_U\varphi^\prime (h)^2|\nabla u|^2 w^2\right)^{1/2}+
2/\nu \left(\int_U\varphi^\prime (h)^2|\nabla\theta|^2 w^2\right)^{1/2}.\nonumber
\end{align}
Let $\delta\in (0,1/1000)$ and suppose that $\nu$ is so 
large that $\norm{h}_{\infty}\le t_{\delta}$, where $t_{\delta}>0$ is a constant such that 
\eqref{(2.18)} and \eqref{(2.19)} hold for all $t\in (0,t_{\delta}]$ with $p=2$.
Then by the Caccioppoli inequality~\eqref{cacciomin}, \eqref{(2.18)}, and \eqref{(2.19)}, 
the first term on the right-side of \eqref{(4.10)} can be estimated from above as
\begin{align*}
&\quad\ 2/\nu\left(\int_U\varphi^\prime (h)^2|\nabla u|^2 w^2\right)^{1/2} \le 
 \sqrt{2}(1+\delta)/\nu\left(\int_U\psi^\prime (h)|\nabla u|^2 w^2\right)^{1/2}\\
 &\le \sqrt{2}(1+\delta)/\nu\biggl(4\int_U \psi^\prime(h)|\nabla\theta|^2 w^2 
+8\nu^2\int_U \frac{\psi^2}{\psi^\prime}(h)|\nabla w|^2\\
&\quad +4\nu^2\int_U \frac{\psi^2}{\psi^\prime}(h)|\nabla\log W|^2 w^2\biggr)^{1/2}\\
&\le \sqrt{2}(1+\delta)/\nu\biggl(4\int_U \psi^\prime(h)|\nabla\theta|^2 w^2 
+4\nu^2(1+\delta)^2\int_U\varphi(h)^2|\nabla w|^2\\
&\quad +2\nu^2(1+\delta)^2\int_U \varphi(h)^2|\nabla\log W|^2 w^2\biggr)^{1/2}\\
&\le \sqrt{2}(1+\delta)/\nu\biggl(16\int_U\varphi^\prime(h)^2|\nabla\theta|^2 w^2
+4\nu^2(1+\delta)^2\int_U\varphi(h)^2|\nabla w|^2\\
&\quad +2\nu^2(1+\delta)^2\int_U \varphi(h)^2|\nabla\log W|^2 w^2\biggr)^{1/2}\\
&\le 4\sqrt{2}(1+\delta)/\nu\left(\int_U \varphi^\prime(h)^2 |\nabla\theta|^2 w^2\right)^{1/2}
+\sqrt{8}(1+\delta)^2\left(\int_U\varphi(h)^2|\nabla w|^2\right)^{1/2}\\
&\quad +2(1+\delta)^2\left(\int_U \varphi(h)^2|\nabla\log W|^2 w^2\right)^{1/2}.
\end{align*}
Taking into account the upper bounds \eqref{eq-gradw-est} and \eqref{eq-gradw-est2} for
$|\nabla w|$ we obtain 
\[
\int_U\varphi(h)^2|\nabla w|^2\le c + \int_U\varphi(h)^2L(r),
\]
and therefore
\begin{align}\label{(4.11)}
&\quad\ (n-\sqrt{8}(1+\delta)^2) \left(\int_U \varphi(h)^2 L(r) \right)^{1/2}\\
&\le \frac{4\sqrt{2}(1+\delta)+2}{\nu}\left(\int_U\varphi^\prime(h)^2|\nabla\theta|^2 w^2\right)^{1/2} \nonumber\\
&\quad +2(1+\delta)^2\left(\int_U \varphi(h)^2|\nabla\log W|^2 w^2\right)^{1/2} +C. \nonumber
\end{align}
Next we apply the complementary Young functions $F(\sqrt{\cdot})$ and $G(\sqrt{\cdot})^2$ as in the proof 
of Lemma~\ref{mainlemma} to estimate the first term on the right-side of \eqref{(4.11)}
\begin{align*}
\int_U\varphi^\prime(h)^2|\nabla\theta|^2 w^2
&= \int_U \varphi^\prime(h)^2 L(r)\left(\frac{|\nabla\theta|r\log(1+r)}{L(r)}\right)^2\\
&\le k\int_U G\bigl(\varphi^\prime (h)\bigr)^2L(r)+k\int_U F\left(\frac{|\nabla\theta|r\log(1+r)}{\sqrt{k}L(r)}\right)L(r)\\
&=k\int_U \varphi(h)^2 L(r) +k\int_U F\left(\frac{|\nabla\theta|r\log(1+r)}{\sqrt{k}L(r)}\right)L(r)
\end{align*}
for all $k>0$. By the assumption $|\nabla\log W|=o(1/r)$ we may estimate the second term on the 
right-side of \eqref{(4.11)} as
\begin{align*}
\int_U \varphi(h)^2|\nabla\log W|^2 w^2 & =\int_U\varphi(h)^2 L(r)
\left(\frac{|\nabla\log W|r\log(1+r)}{\log(1+r)+\cC(r)}\right)^2\\
&\le \delta\int_U\varphi(h)^2 L(r) +C_\delta.
\end{align*}
Choosing $k>0$ small enough and $c_4=\nu$ large enough we finally obtain \eqref{2ndphiFest}.
\end{proof}

\subsection{Solving the asymptotic Dirichlet problem with Lipschitz boundary values}
Since the asymptotic boundary $\partial_\infty M$ is homeomorphic to the unit sphere $\Sph^{n-1}\subset T_o M$, 
we may interpret the given boundary value function $f\in\linebreak C(\partial_\infty M)$ as a continuous function on 
$\Sph^{n-1}$.
In this section we solve the asymptotic Dirichlet problem for \eqref{Mequ} with Lipschitz continuous boundary values
$f\in C(\Sph^{n-1})$. First we construct an extension of $f$ as in \cite{Ho}. We assume that, for all $x\in M$ and for all $2$-planes $P\subset T_x M$,
\begin{equation}\label{a-b-curv}
-b^2\bigl(r(x)\bigr)\le K(P) \le -a^2\bigl(r(x)\bigr),
\end{equation}
where $a,b\colon [0,\infty)\to [0,\infty)$ are smooth functions that are constant in some neighborhood of $0$ 
and
\begin{align*}
a^2(t)&=\frac{1+\varepsilon}{t^{2}\log t},\\
b^2(t)&=\frac{\left(\log t\right)^{2\bar{\varepsilon}}}{t^2}
\end{align*}
for $t\ge R_0$. We identify $\partial_\infty M$ with the unit sphere 
$\Sph^{n-1}\subset T_o M$ and assume that $f\colon \Sph^{n-1}\to \R$ is $L$-Lipschitz. 
We extend $f$ radially to a 
continuous function $\tilde{\theta}$ on $M\setminus \{o\}$. The Lipschitz continuity of $f$ and the curvature upper bound 
imply that
\[
\osc (\tilde{\theta}, B(x,3))\leq \dfrac{cL}{f_a (r(x))}
\]
for $x\in M\setminus\{o\}$, where $f_a$ is the solution to the Jacobi equation \eqref{Jacobi_eq}. Next we will define a smooth function $\theta$ on $M$ such that
\begin{equation}\label{limit-f}
\lim_{x\to \xi} \theta(x)=f(\xi),
\end{equation}
for every $\xi \in \partial_\infty M$ and that  first and second order derivatives of $\theta$ are controlled. 
In order to construct $\theta$, we first fix a maximal $1$-separated set 
$Q=\{q_1,q_2,\ldots\}\subset M\setminus\{o\}$. 
For each $x\in M$, we write $Q_x=Q\cap B(x,3)$. The curvature lower bound implies that
\[
\card Q_x\leq c
\]
for some constant $c$ independent of $x$. We then define $\theta$ as
\begin{equation}\label{def-sm-ext}
\theta(x)=\sum_{q_i \in Q} \tilde{\theta}(q_i)\varphi_i (x),
\end{equation}
where $\{\varphi_i\}$ is a partition of the unity subordinate to $\{B(q_i, 3)\}$ defined as follows. First fix a 
$C^{\infty}$-function $\zeta\colon [0,\infty)\to [0,1]$ such that $\zeta|[0,1]=1,\ \zeta|[2,\infty[=0$, and 
\[
\max\{|\zeta'(t)|,|\zeta''(t)|\}\le c\chi_{[1,2]}(t).
\] 
For $q_i\in Q$ and $x\in M$, let $\eta_i(x)=\zeta\bigl(d(x,q_i)\bigr)$ and finally define
\[
\varphi_i(x)=\frac{\eta_i (x)}{\sum_j \eta_j(x)}.
\]
Following \cite{Ho}, one can easily check that $\theta$ satisfies all the required properties. 
Moreover, the gradient of $\theta$ satisfies
\begin{equation}\label{grad-est-ext}
|\nabla \theta|(x)\leq \dfrac{cL}{f_a (r(x))},
\end{equation}
for all $r(x)\geq 1$.

The next lemma is devoted to proving the decay assumption on\linebreak $\nabla\log W(x)$ used above in
Lemma~\ref{2ndmainlemma}. We will use ideas from \cite[Section 4]{DJX} by Ding, Jost, and Xin. We are grateful to 
J. Spruck for his help to obtain the decay estimate. 
\begin{lemma}\label{W-decay}
Let $M$ be a Cartan-Hadamard manifold satisfying the curvature assumption \eqref{a-b-curv} for all $2$-planes 
$P\subset T_{x}M$. Suppose that $\theta\in C(\bar{M})\cap C^\infty(M)$ is an extension of a Lipschitz function
$f\in C(\partial_\infty M)$ as in \eqref{def-sm-ext}. Let $\Omega=B(o,S)$ and let 
$u\in C^\infty(\Omega)\cap C(\bar{\Omega})$ be the unique solution of \eqref{Mequ} in $\Omega$ with 
$u|\partial\Omega=\theta|\partial\Omega$. Then there exists a continuous function 
$\cW\colon [0,\infty)\to [0,\infty)$ that is independent of $S$
such that $\cW(r)=o(1/r)$ as $r\to\infty$ and
\begin{equation}\label{W-decay-est}
|\nabla\log W(x)|\le\cW\bigl(r(x)\bigr)
\end{equation}
for $x\in \Omega$.
\end{lemma}

\begin{proof}
Since sectional curvatures are bounded from below by a negative constant and $|u|\le \max_{\partial_\infty M}|f|$, 
we have
\[
\max_{\bar{\Omega}}|\nabla u|\le C,
\]
with $C$ independent of the radius $S$. This estimate is obtained by using classical logarithmic type barriers to obtain boundary gradient estimates and then applying \cite[Lemma 3.1]{RSS}. 
In local coordinates $x=(x^1,\ldots, x^n)$ the minimal graph equation can be 
written as
\[
\partial_j \left(\sqrt{\sigma} \frac{\sigma^{ij}u_i}{\sqrt{1+|\nabla u|^2}} \right)=0,
\]
where $\{\partial_j\}$ is the associated coordinate frame, $\sigma_{ij}dx^idx^j$ is the Riemannian metric, 
$\sigma=\det(\sigma_{ij})$, and $(\sigma^{ij})=(\sigma_{ij})^{-1}$.
Differentiating the equation in the direction $\partial_k$ and setting $w=\partial_k u$, we see that $w$ satisfies
\[
L(w)+\partial_j f_k^j=0,
\]
where $L$ is defined by
\[
L(w)= \partial_j \left(\frac{\sigma}{\sqrt{1+|\nabla u|^2}} \left(\sigma^{ij} 
- \frac{u^i u^j}{1+|\nabla u|^2}\right) w_i \right),
\]
with $u^{i}=\sigma^{ij}u_j=\sigma^{ij}\partial_j u$, 
and 
\[
f_k^j= \frac{u_i}{\sqrt{1+|\nabla u|^2}} \partial_k (\sqrt{\sigma} \sigma^{ij}) - \frac{1}{2} \sqrt{\sigma} \sigma^{ij} \frac{u_i u_p u_q}{(1+|\nabla u|^2)^{\frac{3}{2}}} \partial_k \sigma^{pq}.
\]
Fix $p\in M$ and denote $R=d(o,p)$ and 
\[
\rho=\rho(R)=\left(\frac{R}{(\log R)^{\bar{\varepsilon}}}\right)^{2/3}
\]
so that
\[
\tilde{\rho}:=\frac{R}{(\log R)^{\bar{\varepsilon}}\rho}\to\infty
\]
as $R\to\infty$. 
We claim that there are positive constants $\alpha^\prime$, $\theta_1 \in (0,1)$ and $C$ such that there exist  
harmonic coordinates $(x^1,\ldots,x^n)$ on $B(p,\theta_1\rho)$ satisfying 
\begin{align}
\label{harmcoor}
(\sigma_{ij}) &\ge 1/C,\\ 
\sigma_{ij} +\frac{R}{(\log R)^{\bar{\varepsilon}}\rho} |\nabla \sigma_{ij}| &+ 
\left(\frac{(\log R)^{\bar{\varepsilon}}\rho}{R}\right)^{-1-\alpha^\prime} 
\left[ \nabla \sigma_{ij}\right]_{\alpha^\prime, B(p, \theta\rho) } \leq C,\nonumber
\end{align}
where 
\[
\left[\varphi \right]_{\alpha^\prime, B(p,\theta_1\rho)}= 
\sup_{ \stackrel{ \mbox{\scriptsize $x,y\in B(p,\theta_1\rho)$}}{x\ne y}}
\frac{|\varphi (x) - \varphi (y)|}{d(x,y)^{\alpha^\prime}}.
\]
Since we are interested in the asymptotic behavior of $\nabla\log W$ we may assume without loss of 
generality that 
$R$ is so large that $R-\rho\ge R/2\ge R_0$. Hence we have
\[
|\riem|\le \frac{c(\log R)^{2\bar{\varepsilon}}}{R^2}
\]  
for all sectional curvatures in $B(p,\rho)$. By the standard volume comparison theorem we obtain
\[
\vol\bigl(B(p,\rho)\bigr)\leq C \rho^n e^{\frac{(n-1)(\log R)^{\bar{\varepsilon}}\rho}{R}}.
\]
It follows that 
\[
\left\|\riem \right\|_{L^{n/2}(B(p,\rho))}\leq C \left(\frac{(\log R)^{\bar{\varepsilon}}\rho}{R} \right)^2,
\]
and, for $q>n$,
\[
\rho^{2-2n/q} \left\|\ric \right\|_{L^{q/2}(B(p,\rho))}
\leq C \left(\frac{(\log R)^{\bar{\varepsilon}}\rho}{R}\right)^2. 
\]
Then, using these last two estimates, \cite[Theorem 7.1]{DY} applies and gives the existence of the harmonic 
coordinates described above. Using this system of coordinates, we will prove that $\nabla u$ is uniformly H\"older.

Without loss of generality, we may assume that $S$, the radius of $\Omega$, is greater than $2R$.
Let $s\leq \theta_1 \rho/4$ and recall that 
\[
\tilde{\rho}=\frac{R}{(\log R)^{\bar{\varepsilon}}\rho}.
\] 
We define $M_4(s)= \sup_{B(p,4s)} w$, $m_4(s)= \inf_{B(p,4s)} w$, 
$M_1(s)= \sup_{B(p,s)} w$, and $m_1(s)= \inf_{B(p,s)} w$. 
Using \eqref{harmcoor} and the well-known formula for the derivative of the determinant it is easily seen that 
\[
|f_k^j|\leq \frac{C }{\tilde{\rho}}
\] 
on $B(p,\theta_1 \rho)$. Next applying the weak Harnack inequality \cite[Theorem 8.18]{GT}, we have
\begin{equation}
\label{weakharn1}
\frac{1}{s^n} \int_{B(p,2s)} (M_4 (s) - w) \leq C \left(M_4 (s) - M_1 (s)+  s/\tilde{\rho}\right),
\end{equation}
and
\begin{equation}
\label{weakharn2}
\frac{1}{s^n} \int_{B(p,2s)} (w- m_4 (s) ) \leq C \left(m_1 (s) - m_4 (s)+  s/\tilde{\rho}\right).
\end{equation}
Denote $w(s)= M_1(s) -m_1(s)$. Since $\vol \bigl(B(p,2s)\bigr)\geq C_1 s^n$, for some constant $C_1$, using \eqref{weakharn1} 
and \eqref{weakharn2}, we have
\[
C_1 w(4s) \leq \frac{\vol \bigl(B(p,2s))\bigr)}{s^n} w(4s)\leq C (w (4s) - w(s) + 2s/\tilde{\rho}). 
\]
This implies that there exists $\gamma \in (0,1)$ such that, for all $s\in [0, \theta_1 \rho/4]$,
\[
w(s) \leq \gamma w(4s) + 2s/\tilde{\rho}. 
\]
Using \cite[Lemma 8.23]{GT} (notice that $\tilde{\rho}\to\infty$ as $R\to\infty$), we get that there exist 
$\alpha \in (0,1)$ and a positive constant $C$ such that
\[
\|\nabla u \|_{C^\alpha (B(p,\theta_1 \rho))}\leq C \tilde{\rho}^{-\alpha}.
\]
Then the scaling invariant Schauder estimates imply that there exists a constant $C$ depending on $\alpha $ such 
that we have
\begin{equation}
\label{scalinvsch}
\sup_{B(p,\theta_1 \rho/2)} |D^i u|\leq C \rho^{-i} \sup_{B(p,\theta_1 \rho)} |u|,\ \text{for}\ i=1,2.
\end{equation}
Since $\sup_{B(p,\theta_1 \rho)} |u| \leq \max_{\partial_\infty M}|f|$ and
\[
|\nabla\log W|=\frac{|\nabla\langle\nabla u,\nabla u\rangle|}{2\sqrt{1+|\nabla u|^2}}\le 
|\nabla\langle\nabla u,\nabla u\rangle|,
\] the claim \eqref{W-decay-est} follows immediatelly from 
\eqref{scalinvsch} by our choice of
\[
\rho=\left(\frac{R}{(\log R)^{\bar{\varepsilon}}}\right)^{2/3}.\vspace{-2em}
\]
\end{proof}
We are now ready to solve the asymptotic Dirichlet problem with Lipschitz boundary values.
\begin{lemma}\label{lem-sol-lip}
Let $M$ be a Cartan-Hadamard manifold of dimension $n\ge 3$ satisfying the curvature assumption
\eqref{curv_ass_minim2} for all $2$-planes $P\subset T_{x}M$, with $x\in M\setminus B(o,R_0)$. 
Suppose that $f\in C\bigl(\partial_\infty M\bigr)$ is $L$-Lipschitz when interpreted as a function on 
$\Sph^{n-1}\subset T_o M$.
Then the asymptotic Dirichlet problem for the minimal graph equation \eqref{Mequ} is
uniquely solvable with boundary values $f$.
\end{lemma}
\begin{proof}
Let $\theta\in C(\bar{M})\cap C^{\infty}(M)$ be the extension of the given boundary data 
$f\in C(\partial_\infty M)$ defined as above. 
We exhaust $M$ by an increasing sequence of geodesic balls $B_k=B(o,k),\ k\in\N$.
Hence there exist smooth solutions $u_k\in C(\bar{B}_k)$ of the minimal graph equation  
\[
\begin{cases}
\diver \left(\dfrac{\nabla u_k}{\sqrt{1+|\nabla u_k|^2}}\right)=0, &\mbox{in }B_k,\\
u_k\vert\partial B_k=\theta. &
\end{cases}
\]
Then
\[
-\max_{x\in M}\abs{\theta(x)}\le u_k\le \max_{x\in M}\abs{\theta(x)}
\]
in $B_k$ by the comparison principle. 
Standard arguments involving interior gradient estimates \cite[Theorem 1.1]{Spruck} and (regularity) 
theory of elliptic PDEs imply that there exists a subsequence, still denoted by 
$u_k$, that converges in $C^{2}_{\loc}(M)$ to a solution $u\in C^\infty(M)$ of the minimal graph equation. Therefore the proof reduces to prove 
that $u$ extends continuously to $\partial_\infty M$ and satisfies $u\vert \partial_\infty M=f$. 
For each $k$, let $W_k=\sqrt{1+|\nabla u_k|^2}$. Then by Lemma~\ref{W-decay}, 
$|\nabla\log W_k(x)|\le\cW\bigl(r(x)\bigr)$, with $\cW(r)=o(1/r)$ as $r\to\infty$. Applying Lemma \ref{2ndmainlemma} 
and Fatou's lemma and taking into account \eqref{grad-theta-est} we obtain as in the proof of 
Theorem~\ref{thm-A-regu} that
\begin{align}
\int_M \varphi\left(|u-\theta|/\nu\right)^2
&\le\liminf_{k\to\infty}\int_{B(o,k)}\varphi\left(|u-\theta|/\nu\right)^2\\
&\le \nu + \nu \int_{M} F\left(\frac{\nu |\nabla \theta|r\log(1+r)}{L(r)}\right)L(r)<\infty.\nonumber
\end{align}
By Lemma~\ref{moseritemin}, we then get
\[
\lim_{x\to\xi}\sup_{B(x,s/2)}\varphi\left(|u-\theta|/\nu\right)^{2(n+1)}=0
\]
for every $\xi\in\partial_\infty M$. Hence $u$ extends continuously to $\partial_\infty M$ and satisfies $u\vert 
\partial_\infty M=f$.
\end{proof}
\subsection{Solving the Dirichlet problem with continuous boundary values}
\begin{proof}[Proof of Theorem \ref{thmmin}]
Let $f\in C(\partial_\infty M)$. Again we identify $\partial_\infty M$ with the unit sphere 
$\Sph^{n-1}\subset T_o M$. Let $(f_i)$ be a sequence of Lipschitz functions on $\Sph^{n-1}$ such that 
$f_i\to f$ uniformly on $\Sph^{n-1}$. By the previous Lemma \ref{lem-sol-lip} there exist solutions 
$u_i\in C(\bar M)\cap C^\infty (M)$ of the minimal graph equation \eqref{Mequ} with $u_i=f_i$ on 
$\partial_\infty M$. By the maximum principle ,
\[
\sup_{M}|u_i-u_j|=\max_{\partial_\infty M}|f_i -f_j|,
\]
and applying the interior gradient estimate \cite[Theorem 1.1]{Spruck}, we conclude that the 
sequence $(u_i)$ converges in $C(\bar M)\cap C^2_{\loc}(M)$ to a function 
$u\in C(\bar{M})$ that is also a solution to \eqref{Mequ} in $M$ and $u=f$ on $\partial_\infty M$. By regularity theory $u\in C^\infty (M)$. To prove the uniqueness, suppose that $u$ and $v$ are both 
solutions of \eqref{Mequ}, continuous in $\bar{M}$, with $u=v$ on $\partial_\infty M$, and $u(y)>v(y)$ for some 
$y\in M$. Let $\delta=\bigl(u(y)-v(y)\bigr)/2$ and let $U$ be the $y$-component of the set 
$\{x\in M\colon u(x)>v(x)+\delta\}$. Then $U$ is a relatively compact domain and $u=v+\delta$ on $\partial U$.
It follows that $u=v+\delta$ in $U$ which leads to a contradiction since $y\in U$.
\end{proof}


\begin{thebibliography}{10}

\bibitem{ancannals}
A.~Ancona,
\newblock {\em Negatively curved manifolds, elliptic operators, and the {M}artin boundary},
\newblock Ann. of Math. (2) {\bf 125} (1987), no.~3, 495--536.

\bibitem{anchyp}
A.~Ancona,
\newblock {\em Positive harmonic functions and hyperbolicity},
\newblock in: {Potential Theory --- Surveys and Problems (Prague, 1987)},
Vol.~1344 of {Lecture Notes in Math.}, Springer, Berlin, (1988), pp.~1--23.

\bibitem{ancpot}
A.~Ancona,
\newblock {\em Th\'eorie du potentiel sur les graphes et les vari\'et\'es},
\newblock in: {\'Ecole d'\'et\'e de Probabilit\'es de Saint-Flour
  XVIII---1988}, Vol.~1427 of {Lecture Notes in Math.}, Springer, Berlin,
(1990), 1--112.

\bibitem{ancrevista}
A.~Ancona, 
\newblock {\em Convexity at infinity and {B}rownian motion on manifolds with
unbounded negative curvature},
\newblock Rev. Mat. Iberoamericana {\bf 10} (1994), no.~1, 189--220.

\bibitem{andJDG}
 M.~T.~Anderson,
\newblock {\em The {D}irichlet problem at infinity for manifolds of negative curvature},
\newblock J. Differential Geom. {\bf 18} (1983), no.~4, 701--721 (1984).

\bibitem{andschoen}
M.~T.~Anderson and R.~Schoen,
\newblock {\em Positive harmonic functions on complete manifolds of negative curvature},
\newblock Ann. of Math. (2) {\bf 121} (1985), no.~3, 429--461.

\bibitem{Bor}
A.~Borb{\'e}ly,
\newblock {\em The nonsolvability of the {D}irichlet problem on negatively curved manifolds},
\newblock Differential Geom. Appl. {\bf 8} (1998), no.~3, 217--237.

\bibitem{CHH1}
J.-B.~Casteras, E.~Heinonen, and I.~Holopainen,
\newblock {\em Solvability of minimal graph equation under pointwise pinching condition for sectional curvatures},
\newblock J. Geom. Anal. {\bf 27} (2017), no.~2, 1106--1130.

\bibitem{CHR}
J.-B.~Casteras, I.~Holopainen,  and J.~B.~Ripoll,
\newblock {\em On the asymptotic {D}irichlet problem for the minimal hypersurface
equation in a {H}adamard manifold},
\newblock Potential Anal. {\bf 47} (2017), no.~4, 485--501.

\bibitem{ChGrTa}
J.~Cheeger, M.~Gromov, and M.~Taylor,
\newblock{\em Finite propagation speed, kernel estimates for functions of the
{L}aplace operator, and the geometry of complete {R}iemannian manifolds},
\newblock J. Differential Geom. {\bf 17} (1982), no.~1, 15--53.

\bibitem{cheng}
S.~Y.~Cheng,
\newblock {\em The {D}irichlet problem at infinity for non-positively curved
manifolds},
\newblock Comm. Anal. Geom. {\bf 1} (1993), no.~1, 101--112.

\bibitem{choi}
H.~I.~Choi,
\newblock {\em Asymptotic {D}irichlet problems for harmonic functions on
  {R}iemannian manifolds},
\newblock Trans. Amer. Math. Soc. {\bf 281} (1984), no.~2, 691--716.

\bibitem{CR}
P.~Collin and H.~Rosenberg,
\newblock {\em Construction of harmonic diffeomorphisms and minimal graphs},
\newblock Ann. of Math. (2) {\bf 172} (2010), no.~3, 1879--1906.

\bibitem{CoHoSa-Co}
T.~Coulhon, I.~Holopainen, and L.~Saloff-Coste,
\newblock {\em Harnack inequality and hyperbolicity for subelliptic
  {$p$}-{L}aplacians with applications to {P}icard type theorems},
\newblock Geom. Funct. Anal. {\bf 11} (2001), no.~6, 1139--1191.

\bibitem{Croke}
C.~B.~Croke,
\newblock {\em Some isoperimetric inequalities and eigenvalue estimates},
\newblock Ann. Sci. \'Ecole Norm. Sup. (4) {\bf 13} (1980), no.~4, 419--435.

\bibitem{DLR}
M.~Dajczer, J.~H.~de~Lira, and J.~Ripoll,
\newblock {\em An interior gradient estimate for the mean curvature equation of
  {K}illing graphs and applications},
\newblock J. Anal. Math. {\bf 129} (2016), 91--103.

\bibitem{DHL}
M.~Dajczer,  P.~A.~Hinojosa, and J.~H.~de~Lira,
\newblock {\em Killing graphs with prescribed mean curvature},
\newblock Calc. Var. Partial Differential Equations {\bf 33} (2008), no.~2,
  231--248.

\bibitem{DJX}
Q.~Ding, J.~Jost, and Y.~Xin,
\newblock {\em Minimal graphic functions on manifolds of nonnegative {R}icci
  curvature},
\newblock Comm. Pure Appl. Math. {\bf 69} (2016), no.~2, 323--371.

\bibitem{ER}
N.~do~Esp{\'{\i}}rito-Santo and J.~Ripoll,
\newblock {\em Some existence results on the exterior {D}irichlet problem for the
  minimal hypersurface equation},
\newblock Ann. Inst. H. Poincar\'e Anal. Non Lin\'eaire {\bf 28} (2011), no.~3,
  385--393.

\bibitem{EO}
P.~Eberlein and B.~O'Neill,
\newblock {\em Visibility manifolds},
\newblock Pacific J. Math. {\bf 46} (1973), 45--109.

\bibitem{GR}
J.~A.~G{\'a}lvez and H.~Rosenberg,
\newblock {\em Minimal surfaces and harmonic diffeomorphisms from the complex plane
  onto certain {H}adamard surfaces},
\newblock Amer. J. Math. {\bf 132} (2010), no.~5, 1249--1273.

\bibitem{GT}
D.~Gilbarg and N.~S.~Trudinger,
\newblock {Elliptic Partial Differential Equations of Second Order},
\newblock Classics in Mathematics. Springer-Verlag, Berlin, (2001).
\newblock Reprint of the 1998 edition.

\bibitem{HKM}
J.~Heinonen,  T.~Kilpel{\"a}inen,  and O.~Martio,
\newblock {Nonlinear Potential Theory of Degenerate Elliptic Equations},
\newblock The Clarendon Press, Oxford University Press, New York, (1993).
\newblock Oxford Science Publications.

\bibitem{HofSp}
D.~Hoffman and J.~Spruck,
\newblock {\em Sobolev and isoperimetric inequalities for {R}iemannian submanifolds},
\newblock Comm. Pure Appl. Math. {\bf 27} (1974), 715--727.

\bibitem{holDuke}
I.~Holopainen, 
\newblock {\em Volume growth, {G}reen's functions, and parabolicity of ends},
\newblock Duke Math. J. {\bf 97} (1999), no.~2, 319--346.

\bibitem{Ho}
I.~Holopainen,
\newblock {\em Asymptotic {D}irichlet problem for the {$p$}-{L}aplacian on
  {C}artan-{H}adamard manifolds},
\newblock Proc. Amer. Math. Soc. {\bf 130} (2002), no.~11, 3393--3400 (electronic).

\bibitem{H_ns}
I.~Holopainen,
\newblock {\em Nonsolvability of the asymptotic {D}irichlet problem for the
  {$p$}-{L}aplacian on {C}artan-{H}adamard manifolds},
\newblock Adv. Calc. Var. {\bf 9} (2016), no.~2, 163--185.

\bibitem{HR_ns}
I.~Holopainen and J.~B.~Ripoll,
\newblock {\em Nonsolvability of the asymptotic {D}irichlet problem for some
  quasilinear elliptic {PDE}s on {H}adamard manifolds},
\newblock Rev. Mat. Iberoam. {\bf 31} (2015), no.~3, 1107--1129.

\bibitem{HoVa}
I.~Holopainen and A.~V{\"a}h{\"a}kangas,
\newblock {\em Asymptotic {D}irichlet problem on negatively curved spaces},
\newblock J. Anal. {\bf 15} (2007), 63--110.

\bibitem{Hs}
E.~P.~Hsu,
\newblock {\em Brownian motion and {D}irichlet problems at infinity},
\newblock Ann. Probab. {\bf 31} (2003), no.~3, 1305--1319.

\bibitem{Kuf}
A.~Kufner,  O.~John, and S.~Fu{\v{c}}{\'{\i}}k,
\newblock {Function Spaces},
\newblock Noordhoff International Publishing, Leyden; Academia, Prague, (1977).
\newblock Monographs and Textbooks on Mechanics of Solids and Fluids;
  Mechanics: Analysis.

\bibitem{march}
P.~March,
\newblock {\em Brownian motion and harmonic functions on rotationally symmetric
  manifolds},
\newblock Ann. Probab. {\bf 14} (1986), no.~3, 793--801.

\bibitem{MR}
W.~H.~Meeks and H.~Rosenberg,
\newblock {\em The theory of minimal surfaces in {$M\times\Bbb R$}},
\newblock Comment. Math. Helv. {\bf 80} (2005), no.~4, 811--858.

\bibitem{Neel}
R.~W.~Neel,
\newblock {\em Brownian motion and the {D}irichlet problem at infinity on
  two-dimensional {C}artan-{H}adamard manifolds},
\newblock Potential Anal. {\bf 41} (2014), no.~2, 443--462.

\bibitem{NR}
B.~Nelli and H.~Rosenberg,
\newblock {\em Minimal surfaces in {${\Bbb H}^2\times\Bbb R$}},
\newblock Bull. Braz. Math. Soc. (N.S.) {\bf 33} (2002), no.~2, 263--292.

\bibitem{RiSe}
M.~Rigoli and A.~G.~Setti, 
\newblock {\em Liouville type theorems for {$\phi$}-subharmonic functions},
\newblock Rev. Mat. Iberoamericana {\bf 17} (2001), no.~3, 471--520.

\bibitem{RTgeomdedi}
J.~Ripoll and M.~Telichevesky,
\newblock {\em Complete minimal graphs with prescribed asymptotic boundary on
  rotationally symmetric {H}adamard surfaces},
\newblock Geom. Dedicata {\bf 161} (2012), 277--283.

\bibitem{RT}
J.~Ripoll and M.~Telichevesky,
\newblock {\em Regularity at infinity of {H}adamard manifolds with respect to some
  elliptic operators and applications to asymptotic {D}irichlet problems},
\newblock Trans. Amer. Math. Soc. {\bf 367} (2015), no.~3, 1523--1541.

\bibitem{RSS}
H.~Rosenberg, F.~Schulze, and J.~Spruck,
\newblock {\em The half-space property and entire positive minimal graphs in
  {$M\times\Bbb{R}$}},
\newblock J. Differential Geom. {\bf 95} (2013), no.~2, 321--336.

\bibitem{ScYau}
R.~Schoen and S.-T.~Yau,
\newblock {Lectures on Harmonic Maps},
\newblock Conference Proceedings and Lecture Notes in Geometry and Topology,
  II. International Press, Cambridge, MA, (1997).

\bibitem{serrin}
J.~Serrin,
\newblock {\em Local behavior of solutions of quasi-linear equations},
\newblock Acta Math. {\bf 111} (1964), 247--302.

\bibitem{Spruck}
J.~Spruck,
\newblock {\em Interior gradient estimates and existence theorems for constant mean
  curvature graphs in {$M^n\times\bold R$}},
\newblock Pure Appl. Math. Q. {\bf 3} (2007), no.~3 (Special Issue: In honor of Leon
  Simon Part 2), 785--800.

\bibitem{sullivan}
D.~Sullivan,
\newblock {\em The {D}irichlet problem at infinity for a negatively curved manifold},
\newblock J. Differential Geom. {\bf 18} (1983), no.~4, 723--732 (1984).

\bibitem{Va1}
A.~V{\"a}h{\"a}kangas,
\newblock {\em Dirichlet problem at infinity for {${\mathcal A}$}-harmonic
  functions},
\newblock Potential Anal. {\bf 27} (2007), no.~1, 27--44.

\bibitem{Va2}
A.~V{\"a}h{\"a}kangas,
\newblock Dirichlet Problem on Unbounded Domains and at Infinity,
\newblock Reports in {M}athematics, {P}reprint 499, Department of Mathematics
  and Statistics, University of Helsinki, (2009).

\bibitem{DY}
D.~Yang,
\newblock {\em Convergence of {R}iemannian manifolds with integral bounds on
  curvature. {II}},
\newblock Ann. Sci. \'Ecole Norm. Sup. (4) {\bf 25} (1992), no.~2, 179--199.

\end{thebibliography}

\end{document}